\documentclass{amsart}
\usepackage{amsfonts,amssymb,amsmath,amsthm}
\usepackage{url}
\usepackage{enumerate}
\usepackage{epsfig}
\usepackage{tikz}
\usepackage{ulem}
\usetikzlibrary{shapes,arrows}

\newtheorem{theorem}{Theorem}[section]
\newtheorem{lemma}[theorem]{Lemma}
\newtheorem{corollary}[theorem]{Corollary}
\newtheorem{proposition}[theorem]{Proposition}

\theoremstyle{definition}
\newtheorem{definition}[theorem]{Definition}

\theoremstyle{remark}
\newtheorem{remark}[theorem]{Remark}

\numberwithin{equation}{section}

\tikzstyle{block} = [rectangle, draw, text width=5.3em, text centered, rounded corners, minimum height=4em]
\tikzstyle{blocknob} = [text width=4.4em, text centered, minimum height=4em]
\tikzstyle{line} = [draw, -latex']

\def\logmeas {\mathrm{logmeas}\,}
\def\logdens {\mathrm{logdens}\,}
\def\D {\mathcal{D}}
\def\A {\mathcal{A}}

\begin{document}

\title[Wiman-Valiron theory for Askey-Wilson series]{Wiman-Valiron theory for a polynomial series based on the Askey-Wilson operator}


\author[K. H. CHENG]{Kam Hang CHENG}
\address{Department of Mathematics, The Hong Kong University of Science and Technology, Clear Water Bay, Kowloon, Hong Kong S.A.R..}
\email{henry.cheng@family.ust.hk}

\author[Y. M. CHIANG]{Yik-Man CHIANG}
\address{Department of Mathematics, The Hong Kong University of Science and Technology, Clear Water Bay, Kowloon, Hong Kong S.A.R..}
\email{machiang@ust.hk}
\thanks{Both authors were partially supported by GRF no. 16306315 from the Research Grant Council of Hong Kong.\ \ The first author was also partially supported by the PDFS (no. PDFS2021-6S04), and the second author was also partially supported by GRF no. 600609 from the Research Grant Council of Hong Kong.}


\date{August 20, 2020.}

\dedicatory{Dedicated to the memories of Jim Clunie and Walter Hayman}


\begin{abstract}
We establish a Wiman-Valiron theory of a polynomial series based on the Askey-Wilson operator $\D_q$, where $q\in(0,1)$.\ \ For an entire function $f$ of log-order smaller than $2$, this theory includes (i) an estimate which shows that $f$ behaves locally like a polynomial consisting of the terms near the maximal term of its Askey-Wilson series expansion, and (ii) an estimate of $\D_q^n f$ compared to $f$.\ \ We then apply this theory in studying the growth of entire solutions to difference equations involving the Askey-Wilson operator.
\end{abstract}

\maketitle
\setcounter{tocdepth}{2}
\tableofcontents

\section{Introduction}
\label{sec:Intro}

Let $q$ be a complex number with $0<|q|<1$.\ \ The Askey-Wilson operator $\D_q$ was first considered by Askey and Wilson in \cite{Askey-Wilson}.\ \ They constructed a family of basic hypergeometric orthogonal polynomials, now known as the Askey-Wilson polynomials, which are eigensolutions of a second order difference equation in which the divided-difference operator $\D_q$ appears. 

In classical complex function theory, Wiman-Valiron theory has revealed the relationship between an analytic function and the maximal term in its Taylor series expansion.\ \ It has been developed in the 1910s by Wiman \cite{Wiman1, Wiman2}, and then extensively developed by Valiron \cite{Valiron3}, Saxer \cite{Saxer}, Clunie \cite{Clunie1, Clunie2} and K\"{o}vari \cite{Kovari1, Kovari2}.\ \ Hayman has further refined the theory which becomes what is now generally regarded as the K\"{o}vari-Hayman approach \cite{Hayman2}.\ \ Fenton \cite{Fenton1976, Fenton1993, Fenton} has substantially improved previous work by relaxing the finite order restriction and using much more concise arguments.

In recent years, there has been interest to study different topics in special functions related to finite difference calculus using a Nevanlinna-theory approach \cite{Halburd-Korhonen, Chiang-Feng1, Cheng-Chiang, Chiang-Feng3}.\ \ However, when studying function-theoretic properties such as various finite-difference analogues of Wiman-Valiron theories, it appears more natural to make use of a series expansion under an appropriate interpolating polynomial basis \cite{Ismail-Stanton} instead of the usual power series.\ \ Ishizaki and Yanagihara \cite{IY} were the first to study a finite-difference analogue of Wiman-Valiron theory to Newton expansions for entire functions with order smaller than $1/2$, by adopting the K\"{o}vari-Hayman approach \cite{Hayman2} to Netwon interpolation polynomials.\ \ Following this, the first author \cite{Cheng} further modified their theory to Wilson interpolation polynomials, originated from the Wilson operator which is a certain kind of \textit{algebraically deformed} central difference operator.\ \ In this paper, we continue the K\"{o}vari-Hayman approach of the study of Wiman-Valiron-theory based on an interpolating polynomial basis originated from the Askey-Wilson operator $\D_q$, which is both a kind of algebraically deformed central difference operator \textit{and} with a $q$-deformation.\ \ This poses new difficulties in the analysis when compared to previous studies \cite{IY, Cheng}.

Given an entire function $f:\mathbb{C}\to\mathbb{C}$, one defines the \textbf{\textit{logarithmic order}} (or \textbf{\textit{log-order}}) of $f$ to be
\[
	\sigma_{\log}\equiv\sigma_{\log}(f):=\limsup_{r\to\infty}{\frac{\ln^+\ln^+ M(r;f)}{\ln\ln r}},
\]
where $M(r;f):=\max\{|f(x)|:|x|=r\}$.\ \ (Note that a non-constant entire function has log-order at least $1$.)\ \ The main content of this paper is to show that if an entire function $f$ has log-order $\sigma_{\log}<2$, then in its Askey-Wilson series expansion
	\[
		f(x) = \sum_{k=0}^\infty{a_k \phi_k(x;1)},
	\]
	the terms that are ``far away" from a certain ``maximal term" will be small, so that $f$ can be approximated locally by the \textit{polynomial}
	\[
		\sum_{|k-N|< \kappa} a_k \phi_k(x;1),
	\]
	where the integer $N=N(r)\to\infty$ as $r=|x|\to\infty$ is called the \textit{Askey-Wilson central index} of $f$ and the integer $\kappa=\kappa(r)=o(N)$ as $r\to\infty$.\ \ Similar estimates also hold for higher-order Askey-Wilson differences $\mathcal{D}_q^nf$ (Theorem~\ref{WVmain}).\ \ These results conform with the general development of the classical Wiman-Valiron theory for Taylor expansions of entire functions \cite{Hayman2}.\ \ However, the substantial number of technical estimates that are required to establish the current theory for the Askey-Wilson series are not just straight-forward adaptations of those from the classical theory or the previously established interpolation series in \cite{IY, Cheng} where no $q$-deformation was needed.

	Along the way, we have also obtained other related results, among which include (i) a precise formula that relates the log-order $\sigma_{\log}<2$ of an entire function $f$ and the coefficients $\{a_n\}$ of its Askey-Wilson series expansion (Theorem~\ref{AWLP}):
	\[
		\frac{1}{\sigma_{\log}-1}=\liminf_{n\to\infty}{\frac{\ln\ln\frac{1}{|a_n|}}{\ln n}} -1,
	\]
	which mimics the formula that relates the $(p,q)$-order of an entire function (here $(p,q)=(1,2)$) and the coefficients of its Taylor series expansion \cite{JKB}; as well as (ii) a new recursion formula
	\[
	T(k,n) = \frac{q^n+q^{-n}}{2}T(k-1,n) - \frac{q^{1-n}}{2}T(k-1,n-1)
	\]
	for the numbers
	\[
		T(k,n):= \frac{1}{(-2)^n[n]_q^{!}}q^{-\frac{n(n-1)}{4}}\left.\D_q^n x^k\right|_{x=\hat{1}^{(n)}}.
	\]
	(Proposition~\ref{C}).\ \ These numbers are of combinatorial nature, and are needed in establishing some of our main Wiman-Valiron estimates.

This paper is organized as follows.\ \ In \S\ref{sec:Dq}, we will first give the definitions and some basic properties of the Askey-Wilson operator $\D_q$ as well as the Askey-Wilson interpolating polynomial basis $\{\phi_n(x;x_0)\}$.\ \ In the subsequent sections, we will then develop the Wiman-Valiron theory for $\D_q$.\ \ We will state in \S\ref{sec:Mainresults} our main results which include two theorems.\ \ Before proving all these in \S\ref{sec:Proofs}, we will introduce some properties of the Askey-Wilson maximal term and central index in \S\ref{sec:Propmunu}.\ \ Finally, we will apply the main results in \S\ref{sec:Applications} to prove a theorem about the growth of transcendental entire solutions to some linear Askey-Wilson difference equations, and we will strive for some further strengthening of our main results in the discussion section \S\ref{sec:Discussion}.

In this paper, we adopt the following notations:
\begin{enumerate}[(i)]
	\item $\mathbb{N}$ denotes the set of all natural numbers \textit{excluding} $0$, and $\mathbb{N}_0:=\mathbb{N}\cup\{0\}$.
	\item For every positive real number $r$ and every complex number $a$, $D(a;r)$ denotes the open disk of radius $r$ centered at $a$ in the complex plane.
	\item For every positive real number $r$, we denote $\ln^+ r:=\max\{\ln r, 0\}$.
	\item Unless otherwise specified, $q$ always denote a fixed real number with $0<q<1$.\ \ For $n\in\mathbb{N}\cup\{0,\infty\}$, the \textit{$n^{th}$ $q$-shifted factorial} of a complex number $a$ is defined by
	\[
		(a;q)_n:=\prod_{k=0}^{n-1}(1-aq^k),
	\]
	and for complex numbers $a_1,a_2,\ldots,a_k$ we also denote
	\[
		(a_1,a_2,\ldots,a_k;q)_n:=(a_1;q)_n(a_2;q)_n\cdots(a_k;q)_n.
	\]
	For $n\in\mathbb{N}_0$, the \textit{$q$-bracket} of $n$ is defined by
	\[
		[n]_q:=\frac{1-q^n}{1-q}=1+q+q^2+\cdots+q^{n-1},
	\]
	and the \textit{$q$-factorial} of $n$ is defined by
	\[
		[n]_q^{!}:=\prod_{k=1}^n[k]_q = \frac{(q;q)_n}{(1-q)^n}.
	\]
	For integers $0\le k\le n$, the \textit{$q$-binomial coefficient} is defined by
	\[
		\left[\begin{matrix}n\\k\end{matrix}\right]_q:=\frac{[n]_q^{!}}{[k]_q^{!}[n-k]_q^{!}} = \frac{(q;q)_n}{(q;q)_k(q;q)_{n-k}}.
	\]
	\item A \textit{complex function} always means a function in \textit{one} complex variable, and an \textit{entire function} always means a holomorphic function from $\mathbb{C}$ to $\mathbb{C}$, unless otherwise specified.
	\item A summation notation of the form $\displaystyle\sum_{k:S_k}$ denotes a sum running over all the $k$'s such that the statement $S_k$ is true.
	\item For any two functions $f,g:[0,\infty)\to\mathbb{R}$, we write
	\begin{itemize}
		\item $g(r)=O(f(r))$ as $r\to\infty$ if and only if there exist $C>0$ and $M>0$ such that $|g(r)|\le C|f(r)|$ whenever $r>M$;
		\item $g(r)=o(f(r))$ as $r\to\infty$ if and only if for every $C>0$, there exists $M>0$ such that $|g(r)|\le C|f(r)|$ whenever $r>M$;
		\item $f(r)\sim g(r)$ as $r\to\infty$ if and only if for every $\varepsilon>0$, there exists $M>0$ such that $\left|\frac{f(r)}{g(r)}-1\right|<\varepsilon$ whenever $r>M$.
	\end{itemize}
\end{enumerate}

\section{The Askey-Wilson operator and the Askey-Wilson basis}
\label{sec:Dq}

In this section, we give the definition of the Askey-Wilson operator and a few of its properties.

\begin{definition}
	For each $x\in\mathbb{C}$, we write $x=\frac{z+z^{-1}}{2}$ and denote
	\[
		\hat{x} := \frac{q^{\frac12}z+q^{-\frac12}z^{-1}}{2} \hspace{30px} \mbox{and} \hspace{30px} \check{x} := \frac{q^{-\frac12}z+q^{\frac12}z^{-1}}{2}.
	\]
	We also denote $\hat{x}^{(0)}=\check{x}^{(0)}:=x$, $\hat{x}^{(m)}:=(\hat{x}^{(m-1)})\hat{ }$, $\check{x}^{(m)}:=(\check{x}^{(m-1)})\check{ }$, $\hat{x}^{(-m)}=\check{x}^{(m)}$ and $\check{x}^{(-m)}=\hat{x}^{(m)}$ for every positive integer $m$.\ \ Then we define the \textbf{\textit{Askey-Wilson operator}} $\D_q$, which acts on all complex functions, as follows:
	\begin{align}
		\label{A1} (\D_q f)(x) := \frac{f(\hat{x})-f(\check{x})}{\hat{x}-\check{x}} = \frac{f(\frac{q^{\frac12}z+q^{-\frac12}z^{-1}}{2})-f(\frac{q^{-\frac12}z+q^{\frac12}z^{-1}}{2})}{(q^{\frac12}-q^{-\frac12})\frac{z-z^{-1}}{2}}.
	\end{align}
	We also define the \textbf{\textit{Askey-Wilson averaging operator}} $\A_q$ \cite{Ismail1} by
	\begin{align}
		\label{A2} (\A_q f)(x) := \frac{f(\hat{x})+f(\check{x})}{2} = \frac{f(\frac{q^{\frac12}z+q^{-\frac12}z^{-1}}{2})+f(\frac{q^{-\frac12}z+q^{\frac12}z^{-1}}{2})}{2}.
	\end{align}
\end{definition}

According to \eqref{A1} and \eqref{A2}, although there are two choices of $z$ for each $x\ne\pm 1$, $(\D_q f)(x)$ and $(\A_q f)(x)$ are independent of the choice of $z$ and are thus always well-defined.\ \ Moreover, the values of $\D_q f$ at $1$ and $-1$ should be defined respectively as
	\[
		(\D_q f)(\pm 1) := \lim_{x\to\pm 1}(\D_q f)(x) = f'\left(\pm\frac{q^{\frac12}+q^{-\frac12}}{2}\right),
	\]
in case $f$ is differentiable at $\frac{q^{\frac12}+q^{-\frac12}}{2}$ and $-\frac{q^{\frac12}+q^{-\frac12}}{2}$.

We consider the branches of square-root in the expression $z=x+\sqrt{x^2-1}$ with $[-1,1]$ as the branch cut, such that $\sqrt{x^2-1}\approx x$ as $x\to\infty$, and the values of $z$ for $x\in[-1,1]$ is chosen by taking limit as $x$ approaches the segment $[-1,1]$ from above the real axis.\ \ We also have the pointwise limit
\[
	\lim_{q\to 1}(\D_q f)(x) = f'(x)
\]
for each $x\in\mathbb{C}$. \\

We first look at some algebraic properties of the Askey-Wilson operator.\ \ First of all, it is apparent from the definition of $\D_q$ that it is a linear operator, so it makes sense to talk about its kernel.\ \ The second author and Feng \cite[Theorem 10.2]{Chiang-Feng3} have verified that in the space of meromorphic functions, the kernel of $\D_q$ is the field of all functions $f$ of the form
\[
	f(x) = c\prod_{j=1}^k{\frac{(a_jz,a_jz^{-1},\frac{q}{a_j}z,\frac{q}{a_j}z^{-1};q)_\infty}{(b_jz,b_jz^{-1},\frac{q}{b_j}z,\frac{q}{b_j}z^{-1};q)_\infty}},
\]
where $k$ is a non-negative integer and $a_1,\ldots,a_k,b_1,\ldots,b_k,c$ are complex numbers.\ \ In particular, the only entire functions in $\ker\D_q$ are constant functions. \\

The Askey-Wilson operator has the following product rule and quotient rule.
\begin{lemma}
\textup{(Askey-Wilson product and quotient rules)}
\label{PQRule}
\cite[p. 301]{Ismail1}
	For every pair of complex functions $f$ and $g$, we have
	\begin{enumerate}[(i)]
		\item
			$\displaystyle (\D_q(fg))(x) = (\A_q f)(x)(\D_q g)(x) + (\D_q f)(x)(\A_q g)(x)$, and
		\item
			$\displaystyle \left(\D_q\frac{f}{g}\right)(x) = \frac{(\D_q f)(x)(\A_q g)(x) - (\A_q f)(x) (\D_q g)(x)}{g(\hat{x})g(\check{x})}$ whenever $g\not\equiv 0$.
	\end{enumerate}
\end{lemma}

More generally, Ismail came up with the following Leibniz rule for the Askey-Wilson operator.

\begin{theorem}
\textup{(Askey-Wilson Leibniz rule)}
\label{AWLeibniz}
\cite{Ismail}
For every pair of complex functions $f$ and $g$ and every $n\in\mathbb{N}_0$, we have
\[
	\D_q^n(fg) = \sum_{k=0}^n{\left[\begin{matrix}n\\k\end{matrix}\right]_q q^{-\frac{k(n-k)}{2}}(\eta_q^k\D_q^{n-k}f)(\eta_q^{-(n-k)}\D_q^k g)},
\]
where $\eta_q$ is the one-sided Askey-Wilson shift operator defined by
\[
	(\eta_q f)(x)=f(\hat{x}).
\]
\end{theorem}

\bigskip
Now we turn to some analytic properties of the Askey-Wilson operator.\ \ We can easily check from the definition of the Askey-Wilson operator $\D_q$ that it sends polynomials to polynomials.\ \ In fact, we have the following.
\begin{proposition}
\label{Mero}
\cite[Theorem 2.1]{Chiang-Feng3}
	Let $f$ be a complex function.\ \ Then
	\begin{enumerate}[(i)]
	\item if $f$ is entire, then $\D_q f$ and $\A_q f$ are also entire; and
	\item if $f$ is meromorphic, then $\D_q f$ and $\A_q f$ are also meromorphic.
	\end{enumerate}
\end{proposition}

\pagebreak
After introducing the Askey-Wilson operator and some of its useful properties, we will look at a series expansion of entire functions in a polynomial basis based on the Askey-Wilson operator.
\begin{definition}
\label{AWbasis}
	Let $x_0\in\mathbb{C}$.\ \ Then the \textbf{\textit{Askey-Wilson basis}} $\{\phi_k(x;x_0):k\in\mathbb{N}_0\}$ is defined as $\phi_0(x;x_0):=1$, and
	\begin{align*}
		\phi_k(x;x_0) := \left(az,\frac{a}{z};q\right)_k &= \prod_{j=0}^{k-1}{(1-2axq^j+a^2q^{2j})} \\
		&= (-2a)^k q^{\frac{k(k-1)}{2}}\prod_{j=0}^{k-1}{\left(x-\frac{aq^j+a^{-1}q^{-j}}{2}\right)}
	\end{align*}
for $k\in\mathbb{N}$, where $a:=x_0+\sqrt{x_0^2-1}$ with the branch of square root chosen as before, so that $x_0=\frac{a+a^{-1}}{2}$.\ \ We emphasize that whenever the Askey-Wilson basis $\{\phi_k(x;x_0):k\in\mathbb{N}_0\}$ is mentioned, the symbol $a$ automatically take the aforesaid meaning.
\end{definition}

Note that the choice of branch of square root in Definition~\ref{AWbasis} ensures that $x_0,\hat{x_0}^{(2)},\hat{x_0}^{(4)},\ldots$ are distinct points.

\bigskip
The Askey-Wilson operator interacts with the Askey-Wilson basis in a similar way as the ordinary differential operator $\frac{d}{dx}$ does with the basis $\{(x-x_0)^k:k\in\mathbb{N}_0\}$.\ \ The following formula is the Askey-Wilson counterpart of the ordinary differention formula $\frac{d}{dx}x^k=kx^{k-1}$ for every positive integer $k$.
\begin{proposition}
\label{PhiRule}
\cite{Askey-Wilson}
	For every $x_0\in\mathbb{C}$ and every $k\in\mathbb{N}$, we have
	\begin{align*}
		(\D_q\phi_k)(x;x_0) = -2a[k]_q\phi_{k-1}(x;\hat{x_0}).
	\end{align*}
\end{proposition}

The domain of a function defined by an Askey-Wilson series is either ``all" (the whole $\mathbb{C}$) or ``nothing" (just some isolated points), as implied by the following theorem.
\begin{theorem}
\label{convergence0}
\textup{\cite[p. 172]{Gelfond}}
Let $\{a_n\}_{n\in\mathbb{N}_0}$ and $\{x_n\}_{n\in\mathbb{N}_0}$ be sequences of complex numbers such that
\[
	\sum_{k=0}^\infty\frac{1}{|x_k|}<+\infty.
\]
If the polynomial series
\[
	\sum_{k=0}^\infty{a_k(x-x_0)\cdots(x-x_{k-1})}
\]
converges at a point $a\in\mathbb{C}\setminus\{x_0,x_1,x_2,\ldots\}$, then it converges uniformly on every compact subset of $\mathbb{C}$.\ \ In particular, given $x_0=\frac{a+a^{-1}}{2}\in\mathbb{C}$ and a sequence $\{a_n\}_{n\in\mathbb{N}_0}$ of complex numbers, the Askey-Wilson series
\[
	\sum_{k=0}^\infty{a_k\phi_k(x;x_0)}
\]
either converges nowhere except at the points $x_0,\hat{x_0}^{(2)},\hat{x_0}^{(4)},\ldots$ or converges uniformly on every compact subset of $\mathbb{C}$.
\end{theorem}

\bigskip
Ismail and Stanton \cite{Ismail-Stanton} have investigated the Askey-Wilson series expansion of an entire function for real $q\in(0,1)$.\ \ They have found out the coefficients in the expansion and obtained a condition for uniform convergence of such a series on compact subsets of $\mathbb{C}$.
\begin{theorem}
\label{AWSeries}
\textup{(Askey-Wilson series expansion)}
\cite[Theorem 3.1]{Ismail-Stanton}
Let $q\in(0,1)$, $x_0\in\mathbb{C}$ and let $f$ be an entire function satisfying
\begin{align}
	\label{A7} \limsup_{r\to\infty}{\frac{\ln^+ M(r;f)}{(\ln r)^2}} < \frac{1}{2\ln q^{-1}}.
\end{align}
Then there exists a unique sequence of complex numbers $\{a_n\}_{n\in\mathbb{N}_0}$, given by
\begin{align}
	\label{A13} a_n=\frac{1}{(-2a)^n[n]_q^{!}}q^{-\frac{n(n-1)}{4}}(\D_q^n f)(\hat{x_0}^{(n)}),
\end{align}
such that the Askey-Wilson series
\[
	\sum_{k=0}^\infty{a_k\phi_k(x;x_0)}
\]
converges uniformly to $f$ on every compact subset of $\mathbb{C}$.
\end{theorem}

The uniqueness statement in Theorem~\ref{AWSeries} together with \eqref{A13} imply that an Askey-Wilson series can only represent functions that satisfy \eqref{A7}.\ \ Thus if $f\not\equiv0$, $f\in\ker{\D_q}$ and $f$ has at least one zero, then
\[
	\limsup_{r\to\infty}{\frac{\ln{M(r;f)}}{(\ln r)^2}} \ge \frac{1}{2\ln q^{-1}},
\]
i.e. $f$ is an entire function of log-order at least $2$, and of type at least $\frac{1}{2\ln q^{-1}}$ in case the log-order is exactly $2$.

\bigskip
The following corollary, also from Ismail and Stanton's \cite{Ismail-Stanton}, is a formula that relates the $n$th Askey-Wilson difference of $f$ at a point and the values taken by $f$ at the nearby interpolation points.\ \ It is equivalent to an earlier formula first introduced by Cooper \cite{Cooper}.

\begin{corollary}
\label{AWExp}
\cite[(3.3)]{Ismail-Stanton}
Let $f$ be an entire function satisfying \eqref{A7}.\ \ Then at each point $x_0\in\mathbb{C}$, we have
\[
	(\D_q^n f)(\hat{x_0}^{(n)}) = \frac{(-2a)^n q^\frac{n(3+n)}{4}}{(1-q)^n}\sum_{j=0}^n{\left[\begin{matrix}n\\j\end{matrix}\right]_q\frac{(-1)^j q^{\frac{j(j-1)}{2}}}{(q^j a^2;q)_j(q^{2j+1}a^2;q)_{n-j}}f(\hat{x_0}^{(2j)})}
\]
for every non-negative integer $n$.\ \ Replacing $z_0$ by $q^{-\frac{n}{2}}z_0$, we also have
\[
	(\D_q^n f)(x_0) = \frac{(-2a)^n q^\frac{n(3-n)}{4}}{(1-q)^n}\sum_{j=0}^n{\left[\begin{matrix}n\\j\end{matrix}\right]_q\frac{(-1)^j q^{\frac{j(j-1)}{2}}}{(q^{j-n}a^2;q)_j(q^{2j+1-n}a^2;q)_{n-j}}f(\hat{x_0}^{(2j-n)})}
\]
for every non-negative integer $n$.
\end{corollary}

Applying Corollary~\ref{AWExp}, we have the following new result about numbers $T(k,n)$ which arises from the action of the Askey-Wilson operator on the monomials $x^k$.\ \ These numbers are of combinatorial nature.

\begin{proposition}
\label{C}
For every pair of non-negative integers $k$ and $n$, we let
\[
	T(k,n):= \frac{1}{(-2)^n[n]_q^{!}}q^{-\frac{n(n-1)}{4}}\left.\D_q^n x^k\right|_{x=\hat{1}^{(n)}}.
\]
Then $T(k,0)=1$ for all $k\in\mathbb{N}_0$, $T(k,n)=0$ for all non-negative integers $k$ and $n$ with $k<n$, and
\begin{align}
	\label{AA} T(k,n)=q^n\left[\frac{1}{(q;q)_n^2}+\sum_{j=1}^n{\frac{(-1)^j q^{\frac{j(j-1)}{2}}(1+q^j)}{(q;q)_{n-j}(q;q)_{n+j}}\left(\frac{q^j+q^{-j}}{2}\right)^k}\right]
\end{align}
for every pair of positive integers $k$ and $n$.\ \ In particular, we have the followings:
\begin{enumerate}[(i)]
\item $T(k,n)$ satisfy the recurrence
\[
	T(k,n) = \frac{q^n+q^{-n}}{2}T(k-1,n) - \frac{q^{1-n}}{2}T(k-1,n-1)
\]
for every pair of positive integers $k$ and $n$.\ \ In particular, $(-1)^nT(k,n)\ge 0$ for every pair of non-negative integers $k$ and $n$.
\item There exists $K>0$ such that for every pair of positive integers $k$ and $n$, we have
\[
	(-1)^n T(k,n) \le \frac{K q^{\frac{n(n+1)}{2}}}{q^{nk}}.
\]
\end{enumerate}
\end{proposition}
\begin{proof}
(i) follows immediately from \eqref{AA}, which is a direct consequence of Corollary~\ref{AWExp} applied to the function $f(x)=x^k$ and the point $x_0=1$.\ \ (ii) is trivial when $k<n$.\ \ When $k\ge n$, \eqref{AA} gives
\[
	(-1)^n T(k,n)\frac{(q;q)_{2n}q^{nk}}{q^{\frac{n(n+1)}{2}}} = \sum_{j=0}^n{(-1)^{n+j}\frac{q^{\frac{j(j-1)}{2}}}{q^{\frac{n(n-1)}{2}}}\theta_j\left[\begin{matrix}2n\\n-j\end{matrix}\right]_q\left(\frac{q^{n-j}+q^{n+j}}{2}\right)^k},
\]
where $\theta_0=1$ and $\theta_j=1+q^j$ for each $j\in\mathbb{N}$, which implies that
\begin{align*}
	(-1)^n T(k,n)\frac{(q;q)_{2n}q^{nk}}{q^{\frac{n(n+1)}{2}}} &\le 2\sum_{j=0}^n{\frac{q^{\frac{j(j-1)}{2}}}{q^{\frac{n(n-1)}{2}}}\left[\begin{matrix}2n\\n-j\end{matrix}\right]_q\left(\frac{q^{n-j}+q^{n+j}}{2}\right)^k} \\
	&= 2\sum_{j=0}^n{q^{\frac{j(j+1)}{2}}\left[\begin{matrix}2n\\j\end{matrix}\right]_q\left(\frac{1+q^{2n-2j}}{2}\right)^k q^{(k-n)j}} \\
	&\le 2\sum_{j=0}^n{q^{\frac{j(j+1)}{2}}\left[\begin{matrix}2n\\j\end{matrix}\right]_q} \\
	&\le 2\sum_{j=0}^{2n}{q^{\frac{j(j+1)}{2}}\left[\begin{matrix}2n\\j\end{matrix}\right]_q} = 2(-q;q)_{2n},
\end{align*}
so (ii) also follows by taking $K=2\frac{(-q;q)_\infty}{(q;q)_\infty}$.
\end{proof}

After studying the Askey-Wilson series expansion of entire functions, we will develop a Wiman-Valiron theory for this series expansion in the upcoming sections.

\pagebreak
\section{Main results}
\label{sec:Mainresults}
We have seen in Theorem~\ref{AWSeries} that an entire function of log-order smaller than $2$ has an Askey-Wilson series expansion at the point $x_0=1$, which converges uniformly to itself on compact subsets of $\mathbb{C}$.\ \ Such an Askey-Wilson series expansion must in particular converge at each $r>0$, so we are able to make the following definition.

\begin{definition}
\label{AWmunu}
Let $f\not\equiv0$ be an entire function satisfying \eqref{A7} with Askey-Wilson series expansion $\displaystyle f(x)=\sum_{k=0}^\infty{a_k\phi_k(x;1)}$.\ \ The \textbf{\textit{Askey-Wilson maximal term}} and \textbf{\textit{Askey-Wilson central index}} of $f$ are respectively the functions $\mu_q(\cdot;f):(0,+\infty)\to(0,+\infty)$ and $\nu_q(\cdot;f):(0,+\infty)\to\mathbb{N}_0$ defined by
\begin{align*}
	\mu_q(r;f)&:= \max_{n\in\mathbb{N}_0}{\max_{x\in\partial D(0;r)}|a_n\phi_n(x;1)|} \\
	&=\max_{n\in\mathbb{N}_0}{|a_n|2^n q^{\frac{n(n-1)}{2}}\prod_{k=0}^{n-1}{\left(r+\frac{q^k+q^{-k}}{2}\right)}}
\end{align*}
and
\[
	\nu_q(r;f):= \max\left\{n\in\mathbb{N}_0:|a_n|2^n q^{\frac{n(n-1)}{2}}\prod_{k=0}^{n-1}\left(r+\frac{q^k+q^{-k}}{2}\right)=\mu_q(r;f)\right\}.
\]
\end{definition}

Here we only focus on Askey-Wilson series expansions at $1$ for simplicity.\ \ If we consider Askey-Wilson series expansions at any $x_0\in\mathbb{C}$, then
\[
	\mu_q(r;f):= \max_{n\in\mathbb{N}_0}{\max_{x\in\partial D(0;r)}|a_n\phi_n(x;x_0)|}
\]
will lie between
\[
	\max_{n}|a_n|(2|a|)^n q^{\frac{n(n-1)}{2}}\prod_{k=0}^{n-1}\left(r+\frac{|\Re a|q^k+|\Re a^{-1}|q^{-k}}{2}\right)
\]
and
\[
	\max_{n}|a_n|(2|a|)^n q^{\frac{n(n-1)}{2}}\prod_{k=0}^{n-1}\left(r+\frac{|a|q^k+|a^{-1}|q^{-k}}{2}\right)
\]
for every $r>0$, and the analysis will be similar.

\bigskip
The following is the main theorem of this paper.\ \ In the case $h=0$, it says that outside a small exceptional set of radii, the terms in the Askey-Wilson series expansion of an entire function that are ``far away" from the maximal term are small.\ \ In other words, the local behavior of an entire function is mainly contributed by those terms in its Askey-Wilson series expansion that are ``near" the maximal term.\ \ Our arguments are based on the analogous theory on Newton series expansion (which relates to the ordinary difference operator) established by Ishizaki and Yanagihara \cite{IY}, but with numerous fine alterations.

\pagebreak
\begin{theorem}
\label{tail}
Let $\displaystyle f(x)=\sum_{k=0}^\infty{a_k \phi_k(x;1)}$ be a transcendental entire function of log-order $\sigma_{\log}<2$, $\gamma\in(1,\frac{1}{\sigma_{\log}-1})$, and $\delta\in(0,1)$.\ \ Then there exists a set $E\subset[1,\infty)$ of zero logarithmic density\footnote{The \textbf{\textit{logarithmic density}} of a set $E\subseteq[1,\infty)$ is defined by \[\logdens E:=\limsup_{r\to\infty}{\frac{\logmeas(E\cap[1,r])}{\ln r}}=\limsup_{r\to\infty}{\frac{1}{\ln r}\int_{E\cap[1,r]}{\frac{1}{x}\,dm}},\]where $m$ is the Lebesgue measure on $\mathbb{R}$.} such that for every $h\ge 0$, $\beta>0$ and $\omega\in(0,\beta)$, we have
\[
	\sum_{k:|k-N|\ge \kappa}{q^{-hk}[k]_q^h|a_k\phi_k(-r;1)|} = o(\mu_q(r;f)q^{-hN}[N]_q^h b(N)^{\frac{\omega-1}{2}})
\]
as $r\to\infty$ and $r\in(0,\infty)\setminus E$, where $N=\nu_q(r;f)$, $b(N):=N^{-1+\delta}$ and $\kappa=\left[\sqrt{\frac{\beta}{b(N)}\ln\frac{1}{b(N)}}\right]$.
\end{theorem}

Applying Theorem~\ref{tail}, we obtain the following asymptotic behavior of successive Askey-Wilson differences of a transcendental entire function of log-order smaller than $2$.

\begin{theorem}
\label{WVmain}
Let $\displaystyle f(x)=\sum_{k=0}^\infty{a_k \phi_k(x;1)}$ be a transcendental entire function of log-order $\sigma_{\log}<2$, $\gamma\in(1,\frac{1}{\sigma_{\log}-1})$, and $\delta\in(0,1)$.\ \ Then there exists a set $E\subset[1,\infty)$ of zero logarithmic density such that for every $n\in\mathbb{N}$, we have
\[
	q^{nN-\frac{n(n+1)}{2}}\left(\frac{x}{[N]_q}\right)^n(\D_q^n f)(x) = f(x) + o(N^{-4(1-\delta)})M(r;f)
\]
as $r\to\infty$ and $r\in(0,\infty)\setminus E$, where $r=|x|$ and $N=\nu_q(r;f)$.
\end{theorem}

Notice that the main difference between the estimates in the classical case and in our Theorem~\ref{tail} and ~\ref{WVmain} is that the corresponding factors $k^n$ and $N^n$ in Hayman \cite{Hayman2} and Ishizaki and Yanagihara \cite{IY} are replaced by our $(q^{-k}[k]_q)^n$ and $(q^{-N}[N]_q)^n$.

\section{Properties of the Askey-Wilson maximal term and central index}
\label{sec:Propmunu}

We start by stating the following lemmas which are about some useful properties of the functions $\mu_q(\cdot;f)$ and $\nu_q(\cdot;f)$.

\begin{lemma}
\label{munuprop}
Let $f$ be a non-constant entire function satisfying \eqref{A7}.\ \ Then
\begin{enumerate}[(i)]
	\item $\mu_q(\cdot;f)$ is continuous everywhere and strictly increasing on $[R,+\infty)$ for some $R>0$.\ \ Also $\displaystyle\lim_{r\to\infty}{\mu_q(r;f)}=+\infty$.
	\item $\nu_q(\cdot;f)$ is a right-continuous non-decreasing piecewise-constant function.\ \ If $f$ is transcendental, then $\displaystyle\lim_{r\to\infty}{\nu_q(r;f)}=+\infty$.
\end{enumerate}
\end{lemma}
The proof is the same as the one in the classical case.\ \ One can see \cite[Lemma 4.1]{Cheng} for a proof of analogous statements for the Wilson series.

\pagebreak
\begin{lemma}
\label{gg}
Let $\displaystyle f(x)=\sum_{k=0}^\infty{a_k\phi_k(x;1)}$ be an entire function of log-order smaller than $2$ and let $\gamma>1$.\ \ Then for each $n\in\mathbb{N}_0$, there exists $K_n>1$ such that
	\[
		|a_n\phi_n(-r;1)| \le K_n M(r;f)
	\]
for every $r\ge e^{n^\gamma}$, and the sequence $\{K_n\}_{n\in\mathbb{N}_0}$ decreases to $1$.
\end{lemma}
\begin{proof}
We let $K_0$ be a large positive real number and let
\[
	K_n:=\left(1+\frac{q^n+q^{-n}}{2e^{n^\gamma}}\right)^{n+1}\left(1-\frac{q^n+q^{-n}}{2e^{n^\gamma}}\right)^{-n-1}
\]
for $n\ge 1$.\ \ Then the sequence $\{K_n\}_{n\in\mathbb{N}_0}$ decreases to $1$.\ \ Now for each $n\in\mathbb{N}_0$ and each $r\ge e^{n^\gamma}$, we have
\[
	a_n = -\frac{q^n}{\pi i}\int_{\partial D(0;r)}{\frac{f(y)}{\phi_{n+1}(y;1)}\,dy}
\]
by Cauchy's Residue Theorem, so
\begin{align*}
	|a_n \phi_n(-r;1)| &\le \frac{q^n}{\pi}\frac{2\pi r M(r;f)}{2^{n+1}q^{\frac{n(n+1)}{2}}(r-1)(r-\frac{q+q^{-1}}{2})\cdots(r-\frac{q^n+q^{-n}}{2})} \\
	&\ \ \ \,\cdot 2^n q^{\frac{n(n-1)}{2}}(r+1)\left(r+\frac{q+q^{-1}}{2}\right)\cdots\left(r+\frac{q^{n-1}+q^{1-n}}{2}\right) \\
	&= \frac{(1)(1+\frac{1}{r})(1+\frac{q+q^{-1}}{2r})\cdots(1+\frac{q^{n-1}+q^{1-n}}{2r})}{(1-\frac{1}{r})(1-\frac{q+q^{-1}}{2r})(1-\frac{q^2+q^{-2}}{2r})\cdots(1-\frac{q^n+q^{-n}}{2r})}M(r;f) \\
	&\le \left(1+\frac{q^n+q^{-n}}{2r}\right)^{n+1}\left(1-\frac{q^n+q^{-n}}{2r}\right)^{-n-1} M(r;f) \\
	&\le K_n M(r;f).
\end{align*}
\end{proof}

\begin{lemma}
\label{muorder}
Let $f$ be a non-constant entire function of log-order $\sigma_{\log}<2$.\ \ Then
\begin{enumerate}[(i)]
	\item $\displaystyle\sigma_{\log}(\mu_q(\cdot;f))=1+\limsup_{r\to\infty}\frac{\ln\nu_q(r;f)}{\ln\ln r}\le\sigma_{\log}$.
	\item In particular, for every $\gamma<\frac{1}{\sigma_{\log}(\mu_q(\cdot;f))-1}$, we have
	\[
		\nu_q(r;f)^\gamma\le \ln r
	\]
	for every sufficiently large $r\in(0,+\infty)$.
\end{enumerate}
\end{lemma}

We note here that the inequality $\sigma_{\log}(\mu_q(\cdot;f))\le\sigma_{\log}$ in Lemma~\ref{muorder} (i) can in fact be improved to an equality as stated in the following theorem, but we delay the proof of this result to \S\ref{sec:Proofs} and only prove Lemma~\ref{muorder} in the meantime.

\begin{theorem}
\label{muorderequal}
Let $f$ be a non-constant entire function of log-order $\sigma_{\log}<2$.\ \ Then
\[
	\sigma_{\log}=\sigma_{\log}(\mu_q(\cdot;f))=1+\limsup_{r\to\infty}\frac{\ln\nu_q(r;f)}{\ln\ln r}.
\]
\end{theorem}

\bigskip
\noindent \textit{Proof of Lemma~\ref{muorder}.}\ \ Let the Maclaurin series expansion of $f$ be $\displaystyle f(x)=\sum_{k=0}^\infty{b_k x^k}$.\ \ Since $\sigma_{\log}<2$, Theorem~\ref{AWSeries} implies that there exists a sequence $\{a_n\}_{n\in\mathbb{N}_0}$ of complex numbers such that
\[
	f(x)=\sum_{k=0}^\infty{a_k \phi_k(x;1)},
\]
and it follows that
\begin{align}
\label{W11}
\begin{aligned}
	a_n&=\frac{1}{(-2)^n[n]_q^{!}}q^{-\frac{n(n-1)}{4}}\sum_{k=n}^\infty{b_k\left.\D_q^n x^k\right|_{x=\hat{1}^{(n)}}}=\sum_{k=n}^\infty{b_k T(k,n)}
\end{aligned}
\end{align}
where the notation $T(k,n)$ is as in Proposition~\ref{C}.\ \ Next let $f^*$ be the function defined by the Askey-Wilson series
\[
	f^*(x):=\sum_{k=0}^\infty{(-1)^k|a_k|\phi_k(x;1)}
\]
and let $\{b_n^*\}_{n\in\mathbb{N}_0}$ be the sequence of real numbers such that $\displaystyle f^*(x)=\sum_{k=0}^\infty{b_k^* x^k}$.\ \ Note that for each pair of positive integers $k\ge n$ we have
\begin{align}
\label{W0}
\begin{aligned}
	&\ \ \ \,\frac{(-1)^n}{n!}\left.\frac{d^n}{dx^n}(-1)^k\phi_k(x;1)\right|_{x=0} \\
	&= 2^k q^{\frac{k(k-1)}{2}}(-1)^n\cdot\left(\mbox{coefficient of $x^{k-n}$ in $\displaystyle\prod_{j=0}^{k-1}{\left(x-\frac{q^j+q^{-j}}{2}\right)}$}\right) \\
	&\le 2^k q^{\frac{k(k-1)}{2}}\left(\frac{q^{k-n}+q^{n-k}}{2}\right)\cdots\left(\frac{q^{k-1}+q^{1-k}}{2}\right)\frac{k!}{n!(k-n)!} \\
	&\le 2^k q^{\frac{(k-n)(k-n-1)}{2}}\frac{k!}{n!(k-n)!}.
\end{aligned}
\end{align}

Now it suffices to prove (i), and we divide the proof into the following steps.\ \ The proof of (ii) is essentially step (3).
\begin{enumerate}
\item We first show that $f^*$ is entire and $\sigma_{\log}(f^*)\le\sigma_{\log}$.\ \ Let $\gamma\in(1,\frac{1}{\sigma_{\log}-1})$ be arbitrary.\ \ Analogous to the Lindel\"{o}f-Pringsheim theorem, we have the following formula obtained by Juneja, Kapoor and Bajpai \cite{JKB}
\[
	\frac{1}{\sigma_{\log}-1}=\liminf_{k\to\infty}{\frac{\ln\ln\frac{1}{|b_k|}}{\ln k}}-1,
\]
so $|b_k|<e^{-k^{1+\gamma}}$ for every sufficiently large $k$.\ \ Applying this together with Proposition~\ref{C} (ii) to \eqref{W11}, we see that there exist positive constants $K_1$ and $K_2$ such that for every sufficiently large $n\in\mathbb{N}$,
\begin{align*}
	|a_n| &\le \sum_{k=n}^\infty{|b_k T(k,n)|} \le K_1 q^{\frac{n(n+1)}{2}}\sum_{k=n}^\infty{\frac{1}{q^{nk}e^{k^{1+\gamma}}}} \\
	&\le \frac{K_1 q^{\frac{n(n+1)}{2}}}{q^{n^2}e^{n^{1+\gamma}}}\sum_{k=0}^\infty{\frac{1}{q^{nk}e^{(n+k)^{1+\gamma}-n^{1+\gamma}}}} \\
	&\le \frac{K_1}{q^{\frac{n(n-1)}{2}}e^{n^{1+\gamma}}}\sum_{k=0}^\infty{\frac{1}{q^{nk}e^{(1+\gamma)n^\gamma k}}} \le \frac{K_2}{q^{\frac{n(n-1)}{2}} e^{n^{1+\gamma}}}.
\end{align*}
Applying \eqref{W0} and Stirling's approximation, we see that there exists a positive constant $K_3$ such that for every sufficiently large $n\in\mathbb{N}$,
\begin{align*}
	|b_n^*| &\le \frac{1}{n!}\sum_{k=n}^\infty{|a_k|\left.\frac{d^n}{dx^n}(-1)^k\phi_k(x;1)\right|_{x=0}} \\
	&\le \frac{K_2}{n!}\sum_{k=n}^\infty{\frac{1}{q^\frac{k(k-1)}{2} e^{k^{1+\gamma}}}2^k q^{\frac{(k-n)(k-n-1)}{2}}\frac{k!}{n!(k-n)!}} \\
	&= \frac{K_2 q^{\frac{n(n+1)}{2}}}{(n!)^2}\sum_{k=n}^\infty{\frac{2^k}{q^{kn}e^{k^{1+\gamma}}}\frac{k!}{(k-n)!}} \le \frac{K_2 q^{\frac{n(n+1)}{2}}}{(n!)^2}\sum_{k=n}^\infty{\frac{2^k k^n}{q^{kn}e^{k^{1+\gamma}}}} \\
	&\le K_3 q^{\frac{n(n+1)}{2}} \frac{e^{2n}}{n^{2n+1}}\frac{2^n n^n}{q^{n^2}e^{n^{1+\gamma}}} = \frac{K_3 (2e^2)^n}{n^{n+1}q^{\frac{n(n-1)}{2}}e^{n^{1+\gamma}}}.
\end{align*}
This shows that $f^*$ is an entire function of log-order
\[
	\sigma_{\log}(f^*) \le 1+\frac{1}{\gamma}.
\]
Since $\gamma\in(1,\frac{1}{\sigma_{\log}-1})$ was arbitrary, we have $\sigma_{\log}(f^*)\le\sigma_{\log}$.

\item We next show that $\sigma_{\log}(\mu_q(\cdot;f))\le\sigma_{\log}(f^*)$.\ \ For every $r>0$, writing $N:=\nu_q(r;f)$ we have
\[
	\mu_q(r;f)=|a_N|2^Nq^{\frac{N(N-1)}{2}}\prod_{k=0}^{N-1}\left(r+\frac{q^k+q^{-k}}{2}\right)\le f^*(-r)\le M(r;f^*),
\]
so we immediately obtain
\[
	\sigma_{\log}(\mu_q(\cdot;f))=\limsup_{r\to\infty}{\frac{\ln\ln\mu_q(r;f)}{\ln\ln r}}\le \limsup_{r\to\infty}{\frac{\ln\ln M(r;f^*)}{\ln\ln r}} =\sigma_{\log}(f^*).
\]

\item Now we show that $\displaystyle 1+\limsup_{r\to\infty}{\frac{\ln\nu_q(r;f)}{\ln\ln r}}\le\sigma_{\log}(\mu_q(\cdot;f))$.\ \ For every $r>e$ and every $R>r$, writing $N:=\nu_q(r;f)$, we have
\begin{align*}
	\left[\frac{R+\frac{q^{N-1}+q^{1-N}}{2}}{r+\frac{q^{N-1}+q^{1-N}}{2}}\right]^N &\le \frac{R+1}{r+1}\frac{R+\frac{q+q^{-1}}{2}}{r+\frac{q+q^{-1}}{2}}\cdots\frac{R+\frac{q^{N-1}+q^{1-N}}{2}}{r+\frac{q^{N-1}+q^{1-N}}{2}} \\
	&= \frac{|\phi_N(-R;0)|}{|\phi_N(-r;0)|} \\
	&\le \frac{\mu_q(R;f)}{\mu_q(r;f)}.
\end{align*}
By Lemma~\ref{munuprop} (i) we have $\mu_q(r;f)\ge 1$ for every sufficiently large $r$, so
\begin{align}
	\label{W8} N\ln\frac{R+\frac{q^{N-1}+q^{1-N}}{2}}{r+\frac{q^{N-1}+q^{1-N}}{2}} \le \ln\mu_q(R;f).
\end{align}
In particular putting
\[
	R=\left(r+\frac{q^{N-1}+q^{1-N}}{2}\right)^{1+\ln\ln r}-\frac{q^{N-1}+q^{1-N}}{2}
\]
in \eqref{W8}, taking natural logarithms and dividing by $\ln\ln r$ on both sides, we get
\begin{align*}
	&\ \ \ \,\frac{\ln N + \ln\ln(r+\frac{q^{N-1}+q^{1-N}}{2}) + \ln\ln\ln r}{\ln\ln r} \\
	&\le \frac{\ln\ln\mu_q(R;f)}{\ln\ln r} \\
	&\le \frac{\ln\ln\mu_q(R;f)}{\ln\ln R}\frac{\ln\ln(r+\frac{q^{N-1}+q^{1-N}}{2}) + \ln(1+\ln\ln r)}{\ln\ln r}.
\end{align*}

Now for every $\gamma<\frac{1}{\sigma_{\log}(\mu_q(\cdot;f))-1}$ and every sufficiently large $r$, we have
\begin{align}
\label{W9}
\begin{aligned}
	&\frac{\ln N + \ln\ln(r+\frac{q^{N-1}+q^{1-N}}{2}) + \ln\ln\ln r}{\ln\ln r} \\
	&\ \ \ \ \ \ \ \,\le \left(1+\frac{1}{\gamma}\right)\frac{\ln\ln(r+\frac{q^{N-1}+q^{1-N}}{2}) + \ln(1+\ln\ln r)}{\ln\ln r}.
\end{aligned}
\end{align}
We claim that $q^{-N}\le r$ for every sufficiently large $r$, so that \eqref{W9} will give
\[
	\frac{\ln N + \ln\ln r+\ln\ln\ln r}{\ln\ln r} \le \left(1+\frac{1}{\gamma}\right)\frac{\ln\ln (2r)+\ln(1+\ln\ln r)}{\ln\ln r}
\]
for every sufficiently large $r$, which implies the desired inequality on taking limit superior as $r\to\infty$.\ \ To prove this claim, we observe that if on the contrary there exists some sequence $\{r_n\}_{n\in\mathbb{N}}$ of positive real numbers increasing to $\infty$ such that $q^{-N_n}>r_n$ for every $n\in\mathbb{N}$, where $N_n:=\nu_q(r_n;f)$, then \eqref{W9} gives
\begin{align*}
	1&<\frac{\ln N_n + \ln\ln q^{-1}}{\ln\ln r_n} \\
	&\le \left(1+\frac{1}{\gamma}\right)\frac{\ln\ln(2q^{-N_n}) + \ln(1+\ln\ln r_n)}{\ln\ln r_n} \\
	&\ \ \ \,+ \frac{\ln\ln q^{-1}-\ln\ln(q^{-N_n}) - \ln\ln\ln r_n}{\ln\ln r_n} \\
	&= \frac{1}{\gamma}\frac{\ln N_n + \ln\ln q^{-1}}{\ln\ln r_n} + o(1)
\end{align*}
as $n\to\infty$, and thus
\begin{align*}
	1-\frac{1}{\gamma}&<\left(1-\frac{1}{\gamma}\right)\frac{\ln N_n + \ln\ln q^{-1}}{\ln\ln r_n}\le o(1)
\end{align*}
as $n\to\infty$, which is a contradiction as $\sigma_{\log}<2$ enables one to choose $\gamma>1$.

\item Finally we show that $\displaystyle\sigma_{\log}(\mu_q(\cdot;f))\le 1+\limsup_{r\to\infty}{\frac{\ln\nu_q(r;f)}{\ln\ln r}}$.\ \ Write $N:=\nu_q(r;f)$ for every $r>0$.\ \ Since $|a_N|\le 1$ for every sufficiently large $r$, we have
\begin{align*}
	&\ \ \ \,\ln\mu_q(r;f) \\
	&= \ln |a_N| +N\ln 2 - \frac{N(N-1)}{2}\ln q^{-1} + \sum_{k=0}^{N-1}{\ln\left(r+\frac{q^k+q^{-k}}{2}\right)} \\
	&=\ln |a_N| +N\ln 2 - \frac{N(N-1)}{2}\ln q^{-1} + N\ln r + \sum_{k=0}^{N-1}{\ln\left(1+\frac{q^k+q^{-k}}{2r}\right)} \\
	&\le N(\ln r+\ln 2) + \sum_{k=0}^{N-1}{\ln\left(1+\frac{q^k+q^{-k}}{2e^{N^\gamma}}\right)} \\
	&\le N(\ln r+\ln 2) + \sum_{k=0}^{N-1}{\ln\left(1+\frac{q^k+q^{-k}}{2e^{k^\gamma}}\right)} \\
	&= N(\ln r + O(1))
\end{align*}
as $r\to\infty$.\ \ This proves the desired inequality.
\end{enumerate}
\hfill \qed

\bigskip
The following is a new Askey-Wilson series analogue of the Lindel\"{o}f-Pringsheim theorem.\ \ It relates the order of the maximal term of an Askey-Wilson series and its coefficients.\ \ In fact one can apply the same technique to obtain a similar result for Newton series under the setting in Ishizaki and Yanagihara's \cite{IY}.
\begin{theorem}
\label{AWLP}
Let $\displaystyle f(x):=\sum_{k=0}^\infty{a_k\phi_k(x;1)}$ be a non-constant entire function of log-order smaller than $2$.\ \ Then
\[
	\frac{1}{\sigma_{\log}-1}=\liminf_{n\to\infty}{\frac{\ln\ln\frac{1}{|a_n|}}{\ln n}} -1.
\]
\end{theorem}
\begin{proof}
We denote $\displaystyle L:=\liminf_{n\to\infty}{\frac{\ln\ln\frac{1}{|a_n|}}{\ln n}}$.
\begin{enumerate}
\item We first show that $L-1\ge\frac{1}{\sigma_{\log}(\mu_q(\cdot;f))-1}$.\ \ By Lemma~\ref{muorder} (i) we have $\sigma_{\log}(\mu_q(\cdot;f))<2<+\infty$, so we let $\gamma\in(1,\frac{1}{\sigma_{\log}(\mu_q(\cdot;f))-1})$ and $\varepsilon\in(0,\gamma-1)$ be arbitrary.\ \ Then for every sufficiently large $n$, we have
\begin{align*}
	&\ \ \ \,|a_n|2^n q^{\frac{n(n-1)}{2}}\prod_{k=0}^{n-1}{\left(e^{n^{\gamma-\varepsilon}}+\frac{q^k+q^{-k}}{2}\right)} \\
	&\le \mu_q(e^{n^{\gamma-\varepsilon}};f) \le e^{(\ln e^{n^{\gamma-\varepsilon}})^{1+\frac{1}{\gamma}}} = e^{n^{1+\gamma-\varepsilon(1+\frac{1}{\gamma})}}.
\end{align*}
So
\begin{align*}
	\ln|a_n| &\le n^{1+\gamma-\varepsilon(1+\frac{1}{\gamma})} -n\ln 2 +\frac{n(n-1)}{2}\ln q^{-1} -\sum_{k=0}^{n-1}{\ln \left(e^{n^{\gamma-\varepsilon}}+\frac{q^k+q^{-k}}{2}\right)} \\
	&\le n^{1+\gamma-\varepsilon(1+\frac{1}{\gamma})} -n\ln 2 +\frac{n(n-1)}{2}\ln q^{-1} - n\ln e^{n^{\gamma-\varepsilon}} \\
	&= -n^{1+\gamma-\varepsilon}(1+o(1))
\end{align*}
as $n\to\infty$. \ \ This gives
\[
	\ln\frac{1}{|a_n|} \ge n^{1+\gamma-\varepsilon}(1+o(1))
\]
as $n\to\infty$, hence
\[
	\frac{\ln\ln\frac{1}{|a_n|}}{\ln n} \ge 1+\gamma-\varepsilon+o(1)
\]
as $n\to\infty$, and therefore $L-1\ge \gamma-\varepsilon$.\ \ Since $\varepsilon\in(0,\gamma-1)$ and $\gamma\in(1,\frac{1}{\sigma_{\log}(\mu_q(\cdot;f))-1})$ are arbitrary, we have $L-1\ge\frac{1}{\sigma_{\log}(\mu_q(\cdot;f))-1}$.
\item Next we show that $L-1\le\frac{1}{\sigma_{\log}(\mu_q(\cdot;f))-1}$.\ \ By the last paragraph and Lemma~\ref{muorder} (i), we have $L-1>1$, so we let $\beta\in(1,L-1)$ be arbitrary.\ \ Then $|a_n|\le e^{-n^{1+\beta}}$ for every sufficiently large $n$.\ \ Now for each $r>0$, since $\beta>1$, we have $n^\beta-\ln 2-n\ln q^{-1}\ge\frac{1}{2}n^\beta\ge \ln r$ for every sufficiently large $n$, and for these $n$ we have $e^{n^\beta}\ge 2q^{\frac{n-1}{2}}(r+q^{-n})$, which gives
\begin{align*}
	&\ \ \ \,|a_n|2^n q^{\frac{n(n-1)}{2}}(r+1)\left(r+\frac{q+q^{-1}}{2}\right)\cdots\left(r+\frac{q^{n-1}+q^{1-n}}{2}\right) \\
	&\le e^{-n^{1+\beta}}2^n q^{\frac{n(n-1)}{2}}(r+1)\left(r+\frac{q+q^{-1}}{2}\right)\cdots\left(r+\frac{q^{n-1}+q^{1-n}}{2}\right) \\
	&\le (r+q^{-n})^{-n}(r+1)\left(r+\frac{q+q^{-1}}{2}\right)\cdots\left(r+\frac{q^{n-1}+q^{1-n}}{2}\right) \le 1.
\end{align*}
Let $\displaystyle a:=\max_{n\in\mathbb{N}_0}{|a_n|}$.\ \ Then for every sufficiently large $r>0$, we have
\begin{align*}
	\mu_q(r;f) &= \max\left\{|a_n|2^n q^{\frac{n(n-1)}{2}}\prod_{k=0}^{n-1}{\left(r+\frac{q^k+q^{-k}}{2}\right)}:n\le(2\ln r)^{\frac{1}{\beta}}\right\} \\
	&\le a\cdot 2^{(2\ln r)^{\frac{1}{\beta}}}r^{(2\ln r)^{\frac{1}{\beta}}+1},
\end{align*}
and so
\[
	\sigma_{\log}(\mu_q(\cdot;f)) = \limsup_{r\to\infty}{\frac{\ln\ln \mu_q(r;f)}{\ln\ln r}} \le 1+\frac{1}{\beta}.
\]
Since $\beta\in(1,L-1)$ is arbitrary, we have $\sigma_{\log}(\mu_q(\cdot;f))\le 1+\frac{1}{L-1}$, and therefore $L-1\le\frac{1}{\sigma_{\log}(\mu_q(\cdot;f))-1}$.
\end{enumerate}
The fact that $\sigma_{\log}=\sigma_{\log}(\mu_q(\cdot;f))$ follows from Theorem~\ref{muorderequal}.
\end{proof}

\section{Proofs of the main results}\label{sec:Proofs}

In the remainder of this paper, we will focus on non-constant entire functions of log-order smaller than $2$ and follow an approach that is similar to \cite{IY} which deals with Newton series expansions, or \cite{Cheng} which deals with Wilson series expansions.\ \ We will show that such an entire function $f$ behaves locally like a polynomial consisting of the few terms around the maximal term in its Askey-Wilson series expansion.\ \ To do this, we write $N:=\nu_q(r;f)$ and aim to show that those terms $a_n\phi_n(x;1)$ in the Askey-Wilson series that are ``far away" from the maximal term $a_N\phi_N(x;1)$ are small, by defining \textit{comparison sequences} $\{\alpha_n\}_n$ and $\{\rho_n\}_n$ and comparing the ratio $\left|\frac{a_n\phi_n(r;1)}{a_N\phi_N(r;1)}\right|$ with $\frac{\alpha_n\rho_N^n}{\alpha_N\rho_N^N}$, whose growth can be controlled.\ \ Please see Hayman \cite{Hayman2} for a comprehensive survey on the modern approach of the Wiman-Valiron method.

\begin{definition}
\label{testseq}
	In the remainder of this paper, we pick a $\delta\in(0,1)$ and define \textbf{\textit{comparison sequences}} $\{\alpha_n\}_{n\in\mathbb{N}_0}$ and $\{\rho_n\}_{n\in\mathbb{N}_0}$ by
	\[
		\alpha_n:=e^{\int_0^n{\alpha(t)dt}} \hspace{20px}\mbox{and}\hspace{20px} \rho_n:=e^{-\alpha(n)},
	\]
	where $\alpha:[0,\infty)\to\mathbb{R}$ is the function
	\[
		\alpha(t):=-\frac{1}{\delta}t^\delta.
	\]
\end{definition}

We immediately have $\rho_0=1<\frac{\alpha_0}{\alpha_1}$ and $\rho_n\in(\frac{\alpha_{n-1}}{\alpha_n},\frac{\alpha_n}{\alpha_{n+1}})$ for every $n\in\mathbb{N}$, so that $\{\rho_n\}_{n\in\mathbb{N}_0}$ is a strictly increasing sequence.

\bigskip
We are interested in only those radii $r$ on which $\left|\frac{a_n\phi_n(r;1)}{a_N\phi_N(r;1)}\right|$ can be controlled by $\frac{\alpha_n\rho_N^n}{\alpha_N\rho_N^N}$, so we give them a name.

\begin{definition}
\label{qnormal}
Let $\displaystyle f(x):=\sum_{k=0}^\infty{a_k \phi_k(x;1)}$ be a non-constant entire function and let $\gamma>1$.\ \ A number $r>0$ is said to be \textbf{\textit{$q$-normal}} (for the Askey-Wilson series $f$, with respect to $\gamma$ and the comparison sequences $\{\alpha_n\}_{n\in\mathbb{N}_0}$ and $\{\rho_n\}_{n\in\mathbb{N}_0}$) if there exists $N\in\mathbb{N}_0$ such that for every $n\in\mathbb{N}_0$,
\begin{align*}
	|a_n \phi_n(r;1)| &\le |a_N \phi_N(r;1)|\frac{\alpha_n}{\alpha_N}\rho_N^{n-N} \hspace{60px}&&\mbox{if $n\ge N$, and} \\
	|a_n \phi_n(r;1)| &\le |a_N \phi_N(r;1)|(1+\varepsilon_{n,N})\frac{\alpha_n}{\alpha_N}\rho_N^{n-N} &&\mbox{if $n<N$},
\end{align*}
where $\varepsilon_{n,N}:=\frac{2q^{-n}}{e^{N^\gamma}}+\cdots+\frac{2q^{1-N}}{e^{N^\gamma}}\le\frac{2q^{1-N}}{e^{N^\gamma}}[N]_q$.\ \ Positive real numbers that are not $q$-normal are said to be \textbf{\textit{$q$-exceptional}}.\ \ Note that if $N$ is sufficiently large (depending only on the choices of $q$ and $\gamma$), we have $\frac{2q^{1-N}}{e^{N^\gamma}}[N]_q<1$, so that $\varepsilon_{n,N}<1$.\ \ We denote by $M_0$ the smallest non-negative integer satisfying this inequality, i.e. $\frac{2q^{1-N}}{e^{N^\gamma}}[N]_q<1$ for every $N\ge M_0$.
\end{definition}

The inequality requirements in Definition~\ref{qnormal} are motivated by the following theorem, which asserts that the majority of non-negative numbers are $q$-normal.

\begin{theorem}
\label{qexceptional}
Let $f$ be a non-constant entire function of log-order $\sigma_{\log}<2$ and let $\gamma\in(1,\frac{1}{\sigma_{\log}-1})$.\ \ Then the set
\[
	E:=\{r\in[1,\infty): \mbox{r is $q$-exceptional for the Askey-Wilson series $f$}\}
\]
has zero logarithmic density.
\end{theorem}
\begin{proof}
The result follows trivially in case $f$ is a polynomial, because every sufficiently large real number is $q$-normal, so that $E$ is bounded.\ \ From now on we assume that $f$ is transcendental.\ \ We write $\displaystyle f(x)=\sum_{k=0}^\infty{a_k \phi_k(x;1)}$.\ \ Since $\nu_q(\cdot;f)$ is integer-valued, non-decreasing and right-continuous by Lemma~\ref{munuprop}, we let $\{r_n\}_{n\in\mathbb{N}_0}$ be the monotonic increasing sequence of non-negative numbers such that $r_0:=0$ and $\nu_q(r;f)=n$ for every $r\in[r_n,r_{n+1})\setminus\{0\}$.\ \ (If $n$ is not in the range of $\nu_q(\cdot;f)$, then $r_{n+1}=r_n$.)

Now by the choice of $\{r_n\}_{n\in\mathbb{N}_0}$ and the continuity of $\mu_q(\cdot;f)$ by Lemma~\ref{munuprop}, for every $j\in\mathbb{N}_0$ and $k\in\mathbb{N}$ satisfying $r_j<r_{j+1}=\cdots=r_{j+k}$, we have
\begin{align*}
	&\ \ \ \,|a_{j+k}|(r_{j+k}+1)\left(r_{j+k}+\frac{q+q^{-1}}{2}\right)\cdots\left(r_{j+k}+\frac{q^{j+k-1}+q^{1-j-k}}{2}\right) \\
	&\le \mu_q(r_{j+k};f) = \mu_q(r_{j+1};f) \\
	&= \lim_{r\to r_{j+1}^-}{\mu_q(r;f)} = \lim_{r\to r_{j+1}^-}{|a_j|(r+1)\left(r+\frac{q+q^{-1}}{2}\right)\cdots\left(r+\frac{q^{j-1}+q^{1-j}}{2}\right)} \\
	&= |a_j|(r_{j+1}+1)\left(r_{j+1}+\frac{q+q^{-1}}{2}\right)\cdots\left(r_{j+1}+\frac{q^{j-1}+q^{1-j}}{2}\right) \\
	&= |a_j|(r_{j+k}+1)\left(r_{j+k}+\frac{q+q^{-1}}{2}\right)\cdots\left(r_{j+k}+\frac{q^{j-1}+q^{1-j}}{2}\right).
\end{align*}
This gives
\begin{align*}
	\frac{|a_{j+k}|}{|a_j|} &\le \frac{1}{(r_{j+k}+\frac{q^j+q^{-j}}{2})\cdots(r_{j+k}+\frac{q^{j+k-1}+q^{1-j-k}}{2})} \\
	&= \frac{1}{(r_{j+1}+\frac{q^j+q^{-j}}{2})\cdots(r_{j+k}+\frac{q^{j+k-1}+q^{1-j-k}}{2})}
\end{align*}
whenever $r_j<r_{j+1}=\cdots=r_{j+k}$.\ \ So for every $n\in\mathbb{N}_0$, taking products for the appropriate $j$'s we get
\begin{align}
\label{W1}\frac{|a_n|}{|a_0|} \le \frac{1}{(r_1+1)(r_2+\frac{q+q^{-1}}{2})\cdots(r_n+\frac{q^{n-1}+q^{1-n}}{2})}.
\end{align}

Since $\rho_n\in(\frac{\alpha_{n-1}}{\alpha_n},\frac{\alpha_n}{\alpha_{n+1}})$ for every $n\in\mathbb{N}$, we have
\begin{align}
\label{W2}\frac{\alpha_n}{\alpha_0}\ge\frac{1}{\rho_1\rho_2\cdots\rho_n}.
\end{align}
Combining \eqref{W1} and \eqref{W2}, we obtain
\[
	\frac{|a_n|}{\alpha_n}\le\frac{|a_0|}{\alpha_0}\frac{\rho_1}{r_1+1}\frac{\rho_2}{r_2+\frac{q+q^{-1}}{2}}\cdots\frac{\rho_n}{r_n+\frac{q^{n-1}+q^{1-n}}{2}}.
\]
Lemma~\ref{muorder} (ii) implies that $r_n>e^{n^\gamma}$ for every sufficiently large $n$, so there exists $K_0>0$ such that
\begin{align}
	\label{W13}A_n:=\frac{|a_n|}{\alpha_n}\le K_0\frac{|a_0|}{\alpha_0}\frac{e^{\frac{1^\delta}{\delta}}}{e^{1^\gamma}}\frac{e^{\frac{2^\delta}{\delta}}}{e^{2^\gamma}}\cdots\frac{e^{\frac{n^\delta}{\delta}}}{e^{n^\gamma}} \le K_0\frac{|a_0|}{\alpha_0}\frac{e^{\frac{1}{\delta}n^{1+\delta}}}{e^{\frac{1}{1+\gamma}n^{1+\gamma}}}
\end{align}
for every sufficiently large $n$.\ \ \eqref{W13} implies that
\[
	\liminf_{n\to\infty}\frac{\ln\ln\frac{1}{|A_n|}}{\ln n} \ge \lim_{n\to\infty}\frac{\ln[\frac{1}{1+\gamma}n^{1+\gamma}-\frac{1}{\delta}n^{1+\delta}-\ln(K_0\frac{|a_0|}{\alpha_0})]}{\ln n} = 1+\gamma,
\]
so by Theorem~\ref{AWLP}, the function $F$ defined by the Askey-Wilson series
\[
	F(x):=\sum_{k=0}^\infty{A_k\phi_k(x;1)}
\]
is a transcendental entire function of log-order at most $1+\frac{1}{\gamma}$.

Now suppose that $\rho>0$ satisfies $M=\nu_q(\rho;F)\ge M_0$, i.e. $\frac{2q^{1-M}}{e^{M^\gamma}}[M]_q<1$ as in Definition~\ref{qnormal} (such $\rho$ exists since $F$ is transcendental, see Lemma~\ref{munuprop} (ii)).\ \ Then noting that $\rho_M>1$, for every $n>M$ we have
\begin{align*}
	\frac{|a_n\phi_n(\rho\rho_M;1)|}{|a_M\phi_M(\rho\rho_M;1)|} &= \frac{\alpha_n A_n|\phi_n(\rho\rho_M;1)|}{\alpha_M A_M|\phi_M(\rho\rho_M;1)|} \\
	&= \frac{\alpha_n A_n}{\alpha_M A_M}\frac{2^n q^{\frac{n(n-1)}{2}}}{2^M q^{\frac{M(M-1)}{2}}}\left(\rho\rho_M +\frac{q^M+q^{-M}}{2}\right)\cdots\left(\rho\rho_M + \frac{q^{n-1}+q^{1-n}}{2}\right) \\
	&\le \frac{\alpha_n A_n}{\alpha_M A_M}\frac{2^n q^{\frac{n(n-1)}{2}}}{2^M q^{\frac{M(M-1)}{2}}}\left(\rho +\frac{q^M+q^{-M}}{2}\right)\cdots\left(\rho + \frac{q^{n-1}+q^{1-n}}{2}\right)\rho_M^{n-M} \\
	&= \frac{\alpha_n A_n|\phi_n(\rho;1)|}{\alpha_M A_M|\phi_M(\rho;1)|}\rho_M^{n-M} \le \frac{\alpha_n}{\alpha_M}\rho_M^{n-M} < 1,
\end{align*}
while for every $n<M$, since $\displaystyle\sum_{k=n}^{M-1}{\frac{2q^{-k}}{e^{M^\gamma}}}\le\frac{2q^{1-M}}{e^{M^\gamma}}[M]_q<1$, we have
\begin{align*}
	&\ \ \ \,\frac{|a_n\phi_n(\rho\rho_M;1)|}{|a_M\phi_M(\rho\rho_M;1)|} \\
	&= \frac{\alpha_n A_n|\phi_n(\rho\rho_M;1)|}{\alpha_M A_M|\phi_M(\rho\rho_M;1)|} \\
	&= \frac{\alpha_n A_n}{\alpha_M A_M}\frac{2^n q^{\frac{n(n-1)}{2}}}{2^M q^{\frac{M(M-1)}{2}}}\frac{1}{(\rho\rho_M +\frac{q^n+q^{-n}}{2})\cdots(\rho\rho_M + \frac{q^{M-1}+q^{1-M}}{2})} \\
	&= \frac{\alpha_n A_n|\phi_n(\rho;1)|}{\alpha_M A_M|\phi_M(\rho;1)|}\rho_M^{n-M}\frac{(\rho\rho_M +\rho_M \frac{q^n+q^{-n}}{2})\cdots(\rho\rho_M + \rho_M\frac{q^{M-1}+q^{1-M}}{2})}{(\rho\rho_M +\frac{q^n+q^{-n}}{2})\cdots(\rho\rho_M + \frac{q^{M-1}+q^{1-M}}{2})} \\
	&\le \frac{\alpha_n}{\alpha_M}\rho_M^{n-M}\frac{(1 + \frac{q^n+q^{-n}}{2\rho})\cdots(1 + \frac{q^{M-1}+q^{1-M}}{2\rho})}{(1 + \frac{q^n+q^{-n}}{2\rho\rho_M})\cdots(1 + \frac{q^{M-1}+q^{1-M}}{2\rho\rho_M})} \\
	&\le \frac{\alpha_n}{\alpha_M}\rho_M^{n-M}\prod_{k=n}^{M-1}{\left(1 + \frac{q^k+q^{-k}}{2\rho}\right)} \\
	&\le \frac{\alpha_n}{\alpha_M}\rho_M^{n-M}\prod_{k=n}^{M-1}{\left(1 + \frac{q^{-k}}{e^{M^\gamma}}\right)} \\
	&\le \frac{\alpha_n}{\alpha_M}\rho_M^{n-M}(1+\varepsilon_{n,M}),
\end{align*}
where in the second last step we have used the inequality $\rho>e^{M^\gamma}$ which follows from Lemma~\ref{muorder} (ii), and in the last step we have used the inequality $\prod_k{(1+\frac{\lambda_k}{2})}\le 1+\sum_k{\lambda_k}$ which holds for every sequence $\{\lambda_k\}_k$ of non-negative numbers with $\sum_k{\lambda_k}<1$.\ \ We have thus shown that $r$ is $q$-normal for $f$ if there exists $\rho>0$ such that $r=\rho\rho_M$ where $M=\nu_q(\rho;F)\ge M_0$, i.e. if there exists $M\ge M_0$ such that $\nu_q(\frac{r}{\rho_M};F)=M$.\ \ Therefore if we let $\{R_n\}_{n\in\mathbb{N}}$ be the monotonic increasing sequence such that $\nu_q(R;F)=n$ for every $R\in[R_n,R_{n+1})$, then
\[
	E\subseteq [1,R_{M_0}\rho_{M_0})\cup\bigcup_{k=M_0+1}^\infty{[R_k\rho_{k-1},R_k \rho_k)}.
\]
Now for every $r\in[R_n \rho_n, R_{n+1}\rho_n)$, we have $r=R\rho_n$ for some $R\in[R_n,R_{n+1})$ and so $\nu_q(r;f)=n$ by the above computations.\ \ So by the definition of $\{r_n\}_{n\in\mathbb{N}_0}$ we have $r_n\le R_n\rho_n$.\ \ Therefore whenever $n\ge M_0$ and $r\in[r_n,r_{n+1})$, i.e. $\nu_q(r;f)=n$, we must have $r<R_{n+2}\rho_{n+1}$, and so
\[
	E\cap[1,r] \subseteq E\cap[1,R_{n+2}\rho_{n+1}) \subseteq [1,R_{M_0}\rho_{M_0})\cup\bigcup_{k=M_0+1}^{n+1}{[R_k\rho_{k-1},R_k \rho_k)},
\]
which implies that
\[
	\logmeas(E\cap[1,r]) \le \int_{1}^{R_{M_0}\rho_{M_0}}{\frac{dt}{t}}+\sum_{k=M_0+1}^{n+1}{\int_{R_k\rho_{k-1}}^{R_k\rho_k}{\frac{dt}{t}}} = \ln R_{M_0} + \ln\rho_{n+1}.
\]
Since $\displaystyle\limsup_{n\to\infty}\frac{\ln\rho_{n+1}}{\ln r_n} \le \lim_{n\to\infty}\frac{(n+1)^\delta}{\delta n^\gamma}= 0$, we have
\[
	\logdens E = \limsup_{r\to\infty}\frac{\logmeas(E\cap[1,r])}{\ln r} \le \limsup_{n\to\infty}\frac{\ln R_{M_0} + \ln\rho_{n+1}}{\ln r_n} = 0.
\]
\end{proof}

We call the set $E$ in Theorem~\ref{qexceptional} the \textit{$q$-exceptional set} for $f$.\ \ We note that $E$ depends not only on $f$, but also on the choice of $\gamma$ as well as the construction of the comparison sequences $\{\alpha_n\}$ and $\{\rho_n\}$ (which depends on the choice of $\delta$).

\bigskip
\begin{lemma}
\label{WVestimate}
Let $\displaystyle f(x)=\sum_{k=0}^\infty{a_k \phi_k(x;1)}$ be a non-constant entire function of log-order $\sigma_{\log}<2$, $\gamma\in(1,\frac{1}{\sigma_{\log}-1})$, and $E$ be the $q$-exceptional set for $f$.\ \ Then for every $r\in(0,\infty)\setminus E$ we have
\[
	\frac{|a_{N+k}\phi_{N+k}(-r;1)|}{\mu_q(r;f)} \le e^{-\frac{1}{2}k^2 b(N+k)}
\]
for every $k\in\mathbb{N}$ and
\[
	\frac{|a_{N-k}\phi_{N-k}(-r;1)|}{\mu_q(r;f)} \le \left[1+\frac{2q^{1-N}}{(1-q)e^{N^\gamma}}\right]e^{-\frac{1}{2}k^2 b(N)}
\]
for every $k\in\{0,1,\ldots,N-1\}$, where $N=\nu_q(r;f)$ and $b(N):=N^{-1+\delta}$.
\end{lemma}
\begin{proof}
From the definition of the comparison sequences $\{\alpha_n\}_{n\in\mathbb{N}_0}$ and $\{\rho_n\}_{n\in\mathbb{N}_0}$, we have
\begin{align*}
	\frac{\alpha_{N+k}}{\alpha_N}\rho_N^k &= e^{\int_N^{N+k}{(\alpha(t)-\alpha(N))\,dt}} = e^{\int_N^{N+k}{(N+k-t)\alpha'(t)\,dt}} \\
	&\le e^{-\frac{1}{2}k^2\min\{|\alpha'(t)|:t\in[N,N+k]\}} = e^{-\frac{1}{2}k^2 b(N+k)}
\end{align*}
for every $k\in\mathbb{N}$, and
\begin{align*}
	\frac{\alpha_{N-k}}{\alpha_N}\rho_N^{-k} &= e^{-\int_{N-k}^N{(\alpha(t)-\alpha(N))\,dt}} = e^{-\int_{N-k}^N{(N-k-t)\alpha'(t)\,dt}} \\
	&\le e^{-\frac{1}{2}k^2\min\{|\alpha'(t)|:t\in[N-k,N]\}} = e^{-\frac{1}{2}k^2 b(N)}
\end{align*}
for every $k\in\{0,1,\ldots,N-1\}$.\ \ So the result follows from Definition~\ref{qnormal}.
\end{proof}

We are now ready to prove Theorem~\ref{muorderequal} and the main Theorem~\ref{tail}.\ \ Note that the following proof of Theorem~\ref{muorderequal} depends only on Lemma~\ref{WVestimate} but not Theorem~\ref{qexceptional}, while in the proof of Theorem~\ref{qexceptional} we have already made use of Theorem~\ref{AWLP} which follows from Theorem~\ref{muorderequal}.

\subsection{Proof of Theorem~\ref{muorderequal}}

The second equality has already been proved in Lemma~\ref{muorder} (i), so it remains to establish the first equality that $\sigma_{\log}(\mu_q(\cdot;f))=\sigma_{\log}$.\ \ Given a non-constant entire function $f$ of log-order $\sigma_{\log}<2$, we let $\gamma\in(1,\frac{1}{\sigma_{\log}-1})$ and let $E$ be the $q$-exceptional set for $f$.\ \ Then for every $\varepsilon>0$, one can deduce that
\begin{align}
\label{Mbound}
	\mu_q(r;f) \le K(r)M(r;f) \le \mu_q(r;f) [\ln\mu_q(r;f)]^{\frac{1-\delta}{2}+\varepsilon}
\end{align}
for every sufficiently large $r\in(0,\infty)\setminus E$, where $K(r):=K_{\nu_q(r;f)}$ is the $\nu_q(r;f)$-th term of the sequence $\{K_n\}$ as defined in the proof of Lemma~\ref{gg}, so that $K(r)$ decreases to $1$ as $r\to\infty$.\ \ The first inequality in \eqref{Mbound} follows from Lemma~\ref{muorder} (ii) and Lemma~\ref{gg}, while the second inequality follows from Lemma~\ref{WVestimate} and similar arguments as in \cite[Theorem 6]{Hayman2} (see also \cite[pp. 330--334]{Hayman2}).\ \ These two inequalities together show that $\sigma_{\log}(\mu_q(\cdot;f))=\sigma_{\log}$.
\hfill \qed

\subsection{Proof of Theorem~\ref{tail}}

The proof is similar to the one of \cite[Theorem 3.3]{IY}.\ \ We take $E$ to be the $q$-exceptional set for $f$.\ \ Then we let $\eta\in(0,\frac12]$ be a number to be determined later, and divide the sum into four parts
\begin{align*}
	&\ \ \ \,\sum_{k:|k-N|\ge \kappa}{q^{-hk}[k]_q^h|a_k\phi_k(-r;1)|} \\
	&= \left(\sum_{k:k\le(1-\eta)N}+\sum_{k:(1-\eta)N<k\le N-\kappa}+\sum_{k:N+\kappa\le k<(1+\eta)N}+\sum_{k:k\ge(1+\eta)N}\right){q^{-hk}[k]_q^h|a_k\phi_k(-r;1)|}.
\end{align*}
\begin{enumerate}[(i)]
	\item For $k\ge(1+\eta)N$ and $r\notin E$, let $p:=k-N$.\ \ Lemma~\ref{WVestimate} gives
		\[
			\frac{q^{-hk}[k]_q^h|a_k\phi_k(-r;1)|}{|a_N \phi_N(-r;1)|} \le e^{-\frac{1}{2}p^2 b(N+p)+(N+p)\ln q^{-h}+h\ln [N+p]_q}.
		\]
		Since $\displaystyle\lim_{r\to\infty}N=+\infty$ by Lemma~\ref{munuprop} and since $p\ge\eta N$, we have
		\begin{align*}
			&\ \ \ \,-\frac{1}{2}p^2 b(N+p)+(N+p)\ln q^{-h}+h\ln [N+p]_q \\
			&\le -\frac{1}{2}\frac{\eta}{1+\eta}p(N+p)^\delta + p\left(\frac{1}{\eta}+1\right)\ln q^{-h} + h\ln\frac{1}{1-q} \\
			&\le -p
		\end{align*}
		for every sufficiently large $r$.\ \ Therefore
		\begin{align*}
			&\ \ \ \,\frac{1}{|a_N \phi_N(-r;1)|}\sum_{k:k\ge(1+\eta)N}{q^{-hk}[k]_q^h|a_k\phi_k(-r;1)|} \\
			&\le \sum_{p:p\ge\eta N}{e^{-p}} \le \int_{\eta N-1}^\infty{e^{-t}\,dt} = e^{1-\eta N}
		\end{align*}
		for every sufficiently large $r\in(0,\infty)\setminus E$.
	\item For $k\le(1-\eta)N$ and $r\notin E$, let $p:=N-k$.\ \ Lemma~\ref{WVestimate} gives
		\[
			\frac{q^{-hk}[k]_q^h|a_k\phi_k(-r;1)|}{|a_N \phi_N(-r;1)|} \le \left[1+\frac{2q^{1-N}}{(1-q)e^{N^\gamma}}\right]e^{-\frac{1}{2}p^2 b(N)+(N-p)\ln q^{-h}+h\ln[N-p]_q}.
		\]
		Since $\displaystyle\lim_{r\to\infty}N=+\infty$ and since $p\ge\eta N$, we have
		\begin{align*}
			&\ \ \ \,-\frac{1}{2}p^2 b(N)+(N-p)\ln q^{-h}+h\ln[N-p]_q \\
			&\le -\frac{1}{2}\eta pN^\delta + p\left(\frac{1}{\eta}-1\right)\ln q^{-h} + h\ln\frac{1}{1-q} \\
			&\le -p
		\end{align*}
		for every sufficiently large $r$.\ \ Therefore similar to the last paragraph we also have
		\[
			\frac{1}{|a_N \phi_N(-r;1)|}\sum_{k:k\le(1-\eta)N}{q^{-hk}[k]_q^h|a_k\phi_k(-r;1)|} \le 2e^{1-\eta N}
		\]
		for every sufficiently large $r\in(0,\infty)\setminus E$.
	\item In the remaining case, we let $\varepsilon\in(0,\frac{2q^{1-N}}{(1-q)e^{N^\gamma}})$ be arbitrary.\ \ Then by the equicontinuity of the family of exponential functions $\{x\mapsto t^x:t\in(0,\frac12]\}$ at $x=1$ and by the continuity of the function $b$, the number $\eta\in(0,\frac{1}{2}]$ can be chosen small enough so that
		\[
			q^{-h\eta}<1+\varepsilon
		\]
		and
		\[
			\left(\frac{1-t^{1+\eta}}{1-t}\right)^h<1+\varepsilon  \hspace{20px} \mbox{for every $t\in(0,\frac12]$}
		\]
		and
		\[
			\frac{b(N+|p|)}{b(N)}>1-\varepsilon \hspace{20px} \mbox{for every $p\in(-\eta N,\eta N)$}.
		\]
		Now for $k\in((1-\eta)N, (1+\eta)N)$ and $r\notin E$, let $p:=k-N$.\ \ Both estimates in Lemma~\ref{WVestimate} give
		\begin{align*}
			&\ \ \ \,\frac{q^{-hk}[k]_q^h|a_k\phi_k(-r;1)|}{|a_N \phi_N(-r;1)|} \\
			&\le q^{-h(1+\frac{p}{N})N}[N+p]_q^h\left[1+\frac{2q^{1-N}}{(1-q)e^{N^\gamma}}\right]e^{-\frac{1}{2}p^2b(N+|p|)} \\
			&\le q^{-hN}(1+\varepsilon)^N[N]_q^h(1+\varepsilon)\left[1+\frac{2q^{1-N}}{(1-q)e^{N^\gamma}}\right]e^{-\frac{1}{2}p^2(1-\varepsilon)b(N)} \\
			&\le \left[1+\frac{2q^{1-N}}{(1-q)e^{N^\gamma}}\right]^{N+2} q^{-hN}[N]_q^h e^{-b^*p^2}
		\end{align*}
		for every $N$ large enough so that $q^{N-1}\in(0,\frac12]$, where $b^*:=\frac{1}{2}(1-\varepsilon)b(N)$.
\end{enumerate}
Now combining the above three paragraphs, we see that for every $\varepsilon\in(0,\frac{2q^{1-N}}{(1-q)e^{N^\gamma}})$ and every $\kappa\in\mathbb{N}$, there exists $\eta\in(0,\frac{1}{2}]$ such that
\begin{align*}
	&\ \ \ \,\sum_{k:|k-N|\ge \kappa}{q^{-hk}[k]_q^h|a_k\phi_k(-r;1)|} \\
	&\le \left[1+\frac{2q^{1-N}}{(1-q)e^{N^\gamma}}\right]^{N+2} q^{-hN}[N]_q^h \mu_q(r;f)\left(2\sum_{p=\kappa}^\infty{e^{-b^*p^2}}+3e^{1-\eta N}\right)
\end{align*}
for every sufficiently large $r\in(0,\infty)\setminus E$.\ \ Note that
\[
	\sum_{p=\kappa}^\infty{e^{-b^*p^2}} \le \int_{\kappa-1}^\infty{e^{-b^*t^2}\,dt} = \frac{1}{\sqrt{b^*}}\int_{y_0}^\infty{e^{-y^2}\,dy} = \frac{1}{\sqrt{b^*}}\left(\frac{e^{-y_0^2}}{2y_0}-\int_{y_0}^\infty{\frac{e^{-y^2}}{2y^2}\,dy}\right)
\]
where $y_0:=(\kappa-1)\sqrt{b^*}$.\ \ So given any $\beta>0$, if we take $\kappa=\left[\sqrt{\frac{\beta}{b(N)}\ln\frac{1}{b(N)}}\right]$, then for each $\omega\in(0,\beta)$, the number $\varepsilon$ can be chosen so small that
\[
	\sum_{p=\kappa}^\infty{e^{-b^*p^2}} = O\left(\frac{e^{-y_0^2}}{y_0\sqrt{b^*}}\right) = O\left(\frac{e^{\frac{1}{2}(1-\varepsilon)\beta\ln b(N)}}{\sqrt{b(N)\ln\frac{1}{b(N)}}}\right) = o(b(N)^{\frac{\omega-1}{2}})
\]
as $r\to\infty$ and $r\in(0,\infty)\setminus E$.
\hfill \qed

\bigskip
\begin{lemma}
\label{polyn}
Let $p$ be a polynomial of degree $d$.\ \ Then as $r\to\infty$, we have
\begin{align}
	\label{W14} \max\{|p(\hat{x})|, |p(\check{x})|\} \le \frac{q^{-\frac{d}{2}}(r+\sqrt{r^2+1})^d}{2^d r^d}\left(1+O\left(\frac{1}{r}\right)\right)M(r;p)
\end{align}
and
\begin{align}
	\label{W15}|(\D_q p)(x)| \le \frac{5\ln q^{-1}}{1-q}\frac{edq^{-\frac{d-1}{2}}(r+\sqrt{r^2+1})^{d-1}}{2^{d-1}r^d}M(r;p)
\end{align}
for every $x\in\partial D(0;r)$.
\end{lemma}
\begin{proof}
Applying maximum principle to $\frac{p(x)}{x^d}$ on $\hat{\mathbb{C}}\setminus D(0;r)$, we have
\[
	|p(y)| \le \frac{|y|^d}{r^d}M(r;p)
\]
whenever $|y|>r$.\ \ Now for every $x\in\partial D(0;r)$, we have
\[
	r+\sqrt{r^2-1}\le\max\{|z|,|z^{-1}|\}\le r+\sqrt{r^2+1},
\]
so for sufficiently large $r$,
\[
	\max\{|\hat{x}|,|\check{x}|\} \le \frac{\frac{q^{\frac12}}{r+\sqrt{r^2+1}}+q^{-\frac12}(r+\sqrt{r^2+1})}{2} = \frac{q^{-\frac12}(r+\sqrt{r^2+1})}{2} + O\left(\frac{1}{r}\right)
\]
as $r\to\infty$.
\begin{enumerate}[(i)]
	\item Using this estimate if $|\hat{x}|\ge r$ or using the maximum principle if $|\hat{x}|< r$, we obtain
	\[
		|p(\hat{x})| = \frac{q^{-\frac{d}{2}}(r+\sqrt{r^2+1})^d}{2^d r^d}\left(1+O\left(\frac{1}{r}\right)\right)M(r;p)
	\]
	as $r\to\infty$.\ \ A similar equality also holds for $|p(\check{x})|$, so \eqref{W14} follows.
	\item Applying Cauchy's inequality we have
	\[
		|p'(y)|\le \frac{edq^{-\frac{d-1}{2}}(r+\sqrt{r^2+1})^{d-1}}{r^d}M(r;p)
	\]
	for every $y\in\overline{D(0;q^{-\frac{1}{2}}(r+\sqrt{r^2+1}))}$ \cite[Lemma 7, p. 337]{Hayman2}.\ \ Hence for every $x\in\partial D(0;r)$,
	\begin{align*}
		&\ \ \ \,|(\D_q p)(x)| = \left|\frac{p(\hat{x})-p(\check{x})}{\hat{x}-\check{x}}\right| \\
		&\le \int_{-\frac{1}{2}}^{\frac{1}{2}}{\left|p'\left(\frac{q^t  z+q^{-t}z^{-1}}{2}\right)\frac{(\ln q)\frac{q^t z-q^{-t}z^{-1}}{2}}{(q^{\frac12}-q^{-\frac12})\frac{z-z^{-1}}{2}}\right|\,dt} \\
		&\le \frac{\ln q^{-1}}{q^{-\frac12}-q^{\frac12}}\left|\frac{q^{-\frac12}(r+\sqrt{r^2+1}) - \frac{q^{\frac12}}{r+\sqrt{r^2-1}}}{(r+\sqrt{r^2-1})-\frac{1}{r+\sqrt{r^2+1}}}\right|\frac{edq^{-\frac{d-1}{2}}(r+\sqrt{r^2+1})^{d-1}}{r^d}M(r;p) \\
		&\le \frac{5\ln q^{-1}}{1-q}\frac{edq^{-\frac{d-1}{2}}(r+\sqrt{r^2+1})^{d-1}}{r^d}M(r;p)
	\end{align*}
	which is \eqref{W15}.
\end{enumerate}
\end{proof}

With Theorem~\ref{tail} and Lemma~\ref{polyn} established, we can now prove Theorem~\ref{WVmain}.

\subsection{Proof of Theorem~\ref{WVmain}}

At each $x\in\mathbb{C}$, we let $b(N):=N^{-1+\delta}$ and $\kappa:=\left[\sqrt{\frac{10}{b(N)}\ln\frac{1}{b(N)}}\right]$, and let
\[
	\psi(x):=\sum_{k:|k-N|>\kappa}{a_k\phi_k(x;1)} \hspace{20px} \mbox{and} \hspace{20px} p(x):=\sum_{k=N-\kappa}^{N+\kappa}{a_k\frac{\phi_k(x;1)}{\phi_{N-\kappa}(x;1)}}.
\]
Then locally $p$ is a polynomial of degree at most $2\kappa$ and
\[
	f(x)=\psi(x) + \phi_{N-\kappa}(x;1)p(x).
\]
We take $E$ to be the $q$-exceptional set for $f$.\ \ To find the asymptotics of $(\D_q^n f)(x)$ outside $E$, we compute the asymptotics of $(\D_q^n\psi)(x)$ and $\D_q^n(\phi_{N-\kappa}(x;1)p(x))$ outside $E$ respectively.
\begin{enumerate}[(i)]
\item Applying Theorem~\ref{tail} with $h=n$, $\beta=10$ and $\omega=9$, we have
\begin{align*}
	&\ \ \ \,r^n|(\D_q^n\psi)(x)| \\
	&= \left|\sum_{k:|k-N|>\kappa}{a_k(-2)^n [k]_q[k-1]_q\cdots[k-n+1]_q r^n\phi_{k-n}(x;{\hat{1}}^{(n)})}\right| \\
	&\le \sum_{k:|k-N|\ge\kappa}{q^{\frac{(k-n)(k-n-1)}{2}-\frac{k(k-1)}{2}}[k]_q^n|a_k\phi_k(x;1)|} \\
	&= \sum_{k:|k-N|\ge\kappa}{q^{\frac{n(n+1)}{2}-nk}[k]_q^n|a_k\phi_k(x;1)|} \\
	&= o(\mu_q(r;f)q^{\frac{n(n+1)}{2}-nN}[N]_q^n b(N)^4) \\
	&= o(\mu_q(r;f)q^{\frac{n(n+1)}{2}-nN}[N]_q^n N^{-4(1-\delta)})
\end{align*}
as $r\to\infty$ and $r\in(0,\infty)\setminus E$.\ \ This gives the asymptotics of $(\D_q^n\psi)(x)$ outside $E$.\ \ In particular, we have
\begin{align}
	\label{W16}|\psi(x)| = o(\mu_q(r;f)N^{-4(1-\delta)}) = o(N^{-4(1-\delta)})M(r;f)
\end{align}
as $r\to\infty$ and $r\in(0,\infty)\setminus E$.

\item To get the asymptotics of $\D_q^n(\phi_{N-\kappa}(x;1)p(x))$ outside $E$, we need some more preparation.\ \ Since $\sum_{k=0}^{N-\kappa-1}\frac{q^k+q^{-k}}{r}<\frac{2q^{1-N}}{e^{N^\gamma}}[N]_{q}<1$ for each sufficiently large $N$ (the first inequality comes from Lemma~\ref{muorder} (ii)), we have
\begin{align*}
	|\phi_{N-\kappa}(x;1)| &\ge (r-1)\left(r-\frac{q+q^{-1}}{2}\right)\cdots\left(r-\frac{q^{N-\kappa-1}+q^{1+\kappa-N}}{2}\right) \\
	&= r^{N-\kappa}\left(1-\frac{1}{r}\right)\left(1-\frac{q+q^{-1}}{2r}\right)\cdots\left(1-\frac{q^{N-\kappa-1}+q^{1+\kappa-N}}{2r}\right) \\
	&\ge r^{N-\kappa}\left(1-\sum_{k=0}^{N-\kappa-1}\frac{q^k+q^{-k}}{r}\right) \\
	&\ge r^{N-\kappa}\left(1-\frac{2q^{1-N}}{e^{N^\gamma}}[N]_{q}\right) \\
	&= r^{N-\kappa}(1-\varepsilon)
\end{align*}
where $\varepsilon\to 0$ as $r\to\infty$.\ \ This together with \eqref{W16} give
\[
	|p(x)|=\frac{1}{|\phi_{N-\kappa}(x;1)|}|f(x)-\psi(x)| \le r^{\kappa-N}(1+\varepsilon')M(r;f),
\]
where $\varepsilon'\to 0$ as $r\to\infty$ and $r\in(0,\infty)\setminus E$.\ \ Setting $M_0:=r^{\kappa-N}(1+\varepsilon')M(r;f)$ and applying \eqref{W15} in Lemma~\ref{polyn}, we have
\[
	|(\D_q p)(x)| \le \frac{5\ln q^{-1}}{1-q}\frac{e(2\kappa)q^{-\frac{2\kappa-1}{2}}(r+\sqrt{r^2+1})^{2\kappa-1}}{r^{2\kappa}}M_0 = O\left(\frac{\kappa(4q^{-1})^\kappa}{r}\right)M_0
\]
as $r\to\infty$ and $r\in(0,\infty)\setminus E$, and inductively for each $k\in\{0,1,\ldots,n\}$ we have
\begin{align}
\label{Dp}
	|(\D_q^k p)(x)| = O\left(\frac{\kappa^k(4q^{-1})^{\kappa k}}{r^k}\right)M_0
\end{align}
as $r\to\infty$ and $r\in(0,\infty)\setminus E$.\ \ We also note that for each $k\in\{0,1,\ldots,n\}$,
\begin{align}
	\label{DPhi}
	\D_q^{n-k}\phi_{N-\kappa}(x;1) &= (-2)^{n-k}\frac{[N-\kappa]_q^{!}}{[N-\kappa-n+k]_q^{!}}\phi_{N-\kappa-n+k}(x;{\hat{1}}^{(n-k)}) \\
	\nonumber &= (-2)^n\frac{[N-\kappa]_q^{!}}{[N-\kappa-n]_q^{!}}\phi_{N-\kappa-n}(x;{\hat{1}}^{(n)}) \\
	\nonumber &\ \ \ \,\cdot\frac{[N-\kappa-n]_q^{!}}{[N-\kappa-n+k]_q^{!}}(-2)^{-k}\frac{\phi_{N-\kappa-n+k}(x;{\hat{1}}^{(n-k)})}{\phi_{N-\kappa-n}(x;{\hat{1}}^{(n)})} \\
	\nonumber &= (-2)^n\frac{[N-\kappa]_q^{!}}{[N-\kappa-n]_q^{!}}\phi_{N-\kappa-n}(x;{\hat{1}}^{(n)})O\left(\frac{[r+q^{-(N-\kappa)}]^k}{q^{-(N-\kappa)k}}\right) \\
	\nonumber &= (-2)^n\frac{[N-\kappa]_q^{!}}{[N-\kappa-n]_q^{!}}\phi_{N-\kappa-n}(x;{\hat{1}}^{(n)})O\left(\frac{r^k}{q^{-(N-\kappa)k}}\right)
\end{align}
as $r\to\infty$, where the last step followed from Lemma~\ref{muorder} (ii).

Now we are ready to compute the asymptotics of $\D_q^n(\phi_{N-\kappa}(x;1)p(x))$ outside $E$.\ \ Applying the Leibniz rule (Theorem~\ref{AWLeibniz}), and then applying \eqref{W14} in Lemma~\ref{polyn} as well as the estimates \eqref{Dp} and \eqref{DPhi} to the relevant polynomials, we obtain
\begin{align*}
	&\ \ \ \,\D_q^n(\phi_{N-\kappa}(x;1)p(x)) \\
	&= \sum_{k=0}^n{\left[\begin{matrix}n\\k\end{matrix}\right]_q q^{-\frac{k(n-k)}{2}}(\eta_q^k\D_q^{n-k}\phi_{N-\kappa})(x;1)(\eta_q^{-(n-k)}\D_q^k p)(x)} \\
	&= \eta_q^{-n} p(x)\D_q^n\phi_{N-\kappa}(x;1) + \sum_{k=1}^n{\left[\begin{matrix}n\\k\end{matrix}\right]_q q^{-\frac{k(n-k)}{2}}(\eta_q^k\D_q^{n-k}\phi_{N-\kappa})(x;1)(\eta_q^{-(n-k)}\D_q^k p)(x)} \\
	&= (-2)^n\frac{[N-\kappa]_q^{!}}{[N-\kappa-n]_q^{!}}\phi_{N-\kappa-n}(x;{\hat{1}}^{(n)}) \\
	&\qquad\cdot\left[q^{-\frac{n}{2}(2\kappa)}p(x)+\sum_{k=1}^n{q^{-\frac{k}{2}(N-\kappa-n+k)}O\left(\frac{r^k}{q^{-(N-\kappa)k}}\right)q^{-\frac{n-k}{2}(2\kappa-k)}O\left(\frac{\kappa^k(4q^{-1})^{\kappa k}}{r^k}\right)M_0}\right] \\
	&= (-2)^n\frac{[N-\kappa]_q^{!}}{[N-\kappa-n]_q^{!}}\phi_{N-\kappa-n}(x;{\hat{1}}^{(n)})q^{-n\kappa}\left(p(x)+\sum_{k=1}^n{O\left(\frac{\kappa^k 4^{\kappa k}}{q^{-\frac{k}{2}(N-\kappa)}}\right)M_0}\right)
\end{align*}
as $r\to\infty$ and $r\in(0,\infty)\setminus E$.\ \ The definition of $\kappa$ implies that the $k=1$ term in the last sum dominates all other terms, so we get
\begin{align*}
	&\ \ \ \,\D_q^n(\phi_{N-\kappa}(x;1)p(x)) \\
	&= (-2)^n\frac{[N-\kappa]_q^{!}}{[N-\kappa-n]_q^{!}}\phi_{N-\kappa-n}(x;{\hat{1}}^{(n)})q^{-n\kappa}\left(p(x)+O\left(\frac{\kappa 4^\kappa}{q^{-\frac{1}{2}(N-\kappa)}}\right)M_0\right)
\end{align*}
as $r\to\infty$ and $r\in(0,\infty)\setminus E$ (note that among the symbols $N$, $\kappa$, $n$ and $k$ appearing in the exponent of $q$, only the first two depend on $r$).\ \ Finally, observing again that $p=f-\psi$ and applying \eqref{W16} to $\psi$, we get
\begin{align*}
	&\ \ \ \,\D_q^n(\phi_{N-\kappa}(x;1)p(x)) \\
	&= (-2)^n\frac{[N-\kappa]_q^{!}}{[N-\kappa-n]_q^{!}}\frac{\phi_{N-\kappa-n}(x;{\hat{1}}^{(n)})}{\phi_{N-\kappa}(x;1)}q^{-n\kappa} \cdot \\
	&\qquad \cdot \left(f(x)-\psi(x)+\phi_{N-\kappa}(x;1)O\left(\frac{\kappa 4^\kappa}{q^{-\frac{1}{2}(N-\kappa)}}\right)M_0\right) \\
	&=(-2)^n\frac{[N-\kappa]_q^{!}}{[N-\kappa-n]_q^{!}}\frac{\phi_{N-\kappa-n}(x;{\hat{1}}^{(n)})}{\phi_{N-\kappa}(x;1)}q^{-n\kappa}\left(f(x)+o(N^{-4(1-\delta)})M(r;f)\right) \\
	&= \frac{[N-\kappa]_q^{!}}{[N-\kappa-n]_q^{!}}\frac{1}{x^n}q^{\frac{n(n+1)}{2}-nN}\left(f(x)+o(N^{-4(1-\delta)})M(r;f)\right)
\end{align*}%
as $r\to\infty$ and $r\in(0,\infty)\setminus E$.\ \ This gives the asymptotics of $\D_q^n(\phi_{N-\kappa}(x;1)p(x))$ outside $E$.
\end{enumerate}
The two asymptotics we obtained in the above two paragraphs imply that
\begin{align*}
	&\ \ \ \,q^{nN-\frac{n(n+1)}{2}}\left(\frac{x}{[N]_q}\right)^n(\D_q^n f)(x) \\
	&= q^{nN-\frac{n(n+1)}{2}}\left(\frac{x}{[N]_q}\right)^n(\D_q^n\psi)(x) + q^{nN-\frac{n(n+1)}{2}}\left(\frac{x}{[N]_q}\right)^n\D_q^n(\phi_{N-\kappa}(x;1)p(x)) \\
	&= f(x)+o(N^{-4(1-\delta)})M(r;f)
\end{align*}
as $r\to\infty$ and $r\in(0,\infty)\setminus E$, as desired.
\hfill \qed

\section{Applications}
\label{sec:Applications}
Our Askey-Wilson version of the Wiman-Valiron theory can be applied when studying difference equations involving the Askey-Wilson operator.\ \ Before we proceed, we have the following remark about Theorem~\ref{WVmain}.

\begin{remark}
\label{WVmainrmk}
We assume the same set-up as in Theorem~\ref{WVmain}.\ \ Then for those $x\in\mathbb{C}$ such that $|f(x)|\ge N^{-4(1-\delta)}M(|x|;f)$, we have
\begin{equation}
\label{hh}
	q^{nN-\frac{n(n+1)}{2}}\left(\frac{x}{[N]_q}\right)^n(\D_q^n f)(x) = (1+o(1))f(x)
\end{equation}
as $|x|\to\infty$ and $|x|\in(0,\infty)\setminus E$.
\end{remark}

Theorem~\ref{WVmain} (in the form of Remark~\ref{WVmainrmk}) can be used to obtain the following result about linear Askey-Wilson difference equations.\ \ The classical analogue of this result about linear differential equations can be found in \cite[\S 4.5]{Valiron3}, and an analogue for $q$-difference equations (using the maximal term and central index regarding \textit{power series}) can be found in \cite{BIY}.
\begin{theorem}
\label{WVDE}
	Let $a_0,\ldots,a_n$ be polynomials with $a_n\not\equiv 0$.\ \ Then every transcendental entire solution to the Askey-Wilson difference equation
	\begin{align}
		\label{WW} a_n\D_q^n y + \cdots + a_1\D_q y + a_0 y = 0,
	\end{align}
	has log-order at least $2$.
\end{theorem}
\begin{proof}
Suppose on the contrary that $f$ is a transcendental entire solution to \eqref{WW} of log-order $\sigma_{\log}<2$.\ \ Then we let
\[
	S:=\{x\in\mathbb{C}:|f(x)|=M(|x|;f)\}.
\]
$S$ has non-empty intersection with $\partial D(0;r)$ for every $r>0$.\ \ Substituting $f$ into \eqref{WW} and applying Theorem~\ref{WVmain} (in the form of \eqref{hh}) to $f$, we have
\[
	\left(a_n q^{\frac{n(n+1)}{2}-nN}\frac{[N]_q^n}{x^n} + \cdots + a_1 q^{1-N}\frac{[N]_q}{x} + a_0\right)f(x)(1+o(1)) = 0
\]
uniformly on $S$ as $r=|x|\to\infty$ and $r\in(0,\infty)\setminus E$, where $N=\nu_q(r;f)$ and $E$ is the $q$-exceptional set for $f$.\ \ So denoting $c_k$ as the leading coefficient of the polynomial $a_k$ for each $k$, we have
\[
	\sum_{k=0}^n{c_k [N]_q^k q^{\frac{k(k+1)}{2}-kN} x^{(\deg a_k) - k}(1+o(1))}=0
\]
uniformly on $S$ as $r\to\infty$ and $r\in(0,\infty)\setminus E$.\ \ This implies that
\begin{align}
	\label{W40}q^{-N}=Lr^\chi(1+o(1))
\end{align}
as $r\to\infty$ and $r\in(0,\infty)\setminus E$ for some $L>0$ and some positive rational number $\chi$ which is the slope of some edge of the Newton polygon for \eqref{WW}, i.e. the convex hull of
	\[
		\bigcup_{k=0}^n\left\{(x,y)\in\mathbb{R}^2: \mbox{$x\ge k$ and $y\le(\deg a_{n-k})-(n-k)$}\right\}.
	\]
	Taking logarithms on both sides of \eqref{W40} we obtain
\[
	\nu_q(r;f) = \frac{\chi}{\ln q^{-1}}(\ln r)(1+o(1))
\]
as $r\to\infty$ and $r\in(0,\infty)\setminus E$.\ \ Now by Lemma~\ref{muorder} (i), we have
\[
	\sigma_{\log}\ge 1+\limsup_{r\to\infty}{\frac{\ln\nu_q(r;f)}{\ln\ln r}} \ge 1+1=2,
\]
which is a contradiction.
\end{proof}

\section{Discussion}
\label{sec:Discussion}

Although the main results in this paper hold for entire functions of log-order smaller than $2$, we believe that it is possible to obtain similar results for all entire functions satisfying \eqref{A7}.\ \ In this direction we consider the \textit{log-type} $\tau_{\log}$ of an entire function $f$, defined by
\[
	\tau_{\log}:=\limsup_{r\to\infty}{\frac{\ln^+ M(r;f)}{(\ln r)^2}},
\]
so that the condition \eqref{A7} is rephrased to that $\displaystyle\tau_{\log}<\frac{1}{2\ln q^{-1}}$.\ \ $\tau_{\log}$ carries useful information only when $\sigma_{\log}=2$, because one always has $\tau_{\log}=0$ if $\sigma_{\log}<2$ and $\tau_{\log}=+\infty$ if $\sigma_{\log}>2$.\ \ Parallel to Lemma~\ref{muorder}, we can obtain the following.

\begin{lemma}
\label{mutype}
Let $f$ be a non-constant entire function of log-type $\displaystyle\tau_{\log}<\frac{1}{2\ln q^{-1}}$.\ \ Then
\begin{enumerate}[(i)]
	\item $\displaystyle\tau_{\log}(\mu_q)\le\frac{1}{\frac{1}{\tau_{\log}}-2\ln q^{-1}}<\infty$.
	\item If $\displaystyle\tau_{\log}(\mu_q)<\frac{1}{4\ln q^{-1}}$, then $\displaystyle\limsup_{r\to\infty}\frac{\nu_q(r;f)}{\ln r}\in[\tau_{\log}(\mu_q), 4\tau_{\log}(\mu_q)]$.\ \ In particular, for every $\displaystyle\gamma<\frac{1}{(4\ln q^{-1})\tau_{\log}(\mu_q)}$, we have
	\[
		q^{-\gamma\nu_q(r;f)}\le r
	\]
	for every sufficiently large $r\in(0,+\infty)$.
\end{enumerate}
\end{lemma}

We also have the following result which can be compared with the formula
\[
	\tau_{\log}=\frac{1}{4}\limsup_{n\to\infty}{\frac{n^2}{\ln\frac{1}{|b_n|}}}
\]
obtained in \cite{JKB2}, about the log-type and the Maclaurin series coefficients $\{b_n\}$ of an entire function.

\begin{theorem}
\label{AWLP2}
Let $\displaystyle f(x):=\sum_{k=0}^\infty{a_k\phi_k(x;1)}$ be a non-constant entire function of log-type $\displaystyle\tau_{\log}<\frac{1}{6\ln q^{-1}}$.\ \ Then
\[
	\frac{1}{\frac{4}{\tau_{\log}(\mu_q)}+2\ln q^{-1}} \le \frac{1}{4}\limsup_{n\to\infty}{\frac{n^2}{\ln\frac{1}{|a_n|}}} \le \frac{1}{\frac{1}{\tau_{\log}(\mu_q)}-2\ln q^{-1}}.
\]
\end{theorem}

For the class of non-constant entire functions having \textit{minimal log-type}, i.e. $\tau_{\log}=0$, Lemma~\ref{mutype} implies that $\tau_{\log}(\mu_q)=0$ also, and that for every real $\gamma$ we have $q^{-\gamma N}\le r$ for every sufficiently large $r$.\ \ Theorem~\ref{AWLP2} implies that $\displaystyle\limsup_{n\to\infty}{\frac{n^2}{\ln\frac{1}{|a_n|}}} =0$.\ \ With the expression of $\varepsilon_{n,N}$ in Definition~\ref{qnormal} modified as
\[
	\varepsilon_{n,N}:=\frac{2q^{-n}}{q^{-\gamma N}}+\cdots+\frac{2q^{1-N}}{q^{-\gamma N}}\le 2q[N]_q q^{(\gamma-1)N},
\]
one can show that the $q$-exceptional set for a non-constant entire function of minimal log-type also has zero logarithmic density, and the main results of this paper (Theorems~\ref{tail} and \ref{WVmain}) extend, with basically the same proofs, to the class of transcendental entire functions having minimal log-type.\ \ Theorem~\ref{WVDE} also asserts that \eqref{WW} has no transcendental entire solutions with minimal log-type.

\bigskip
To summarize, we have adopted the K\"{o}vari-Hayman approach to establish a Wiman-Valiron theory for the algebraically $q$-deformed central difference operator in this paper, following the modification done by Ishizaki-Yanagihara.\ \ However, since the Fenton approach of Wiman-Valiron theory mentioned in \S\ref{sec:Intro} greatly weakens the assumption and clarifies the details needed to establish the theory for the classical Taylor expansions of entire functions, it is natural to ask if one can extend Fenton's approach of Wiman-Valiron theory to cope with the Newton series expansions or Askey-Wilson series expansions of entire functions.

\bigskip
\noindent \textbf{Acknowledgements.}\ \ The authors would like to thank the anonymous referee for her/his helpful and constructive comments and bringing Fenton's work \cite{Fenton} to their attention.

\bibliographystyle{amsplain}

\end{document}